\RequirePackage{fix-cm}
\documentclass[smallextended]{svjour3}       
\smartqed  
\usepackage{graphicx}
\usepackage{amsmath,appendix}
\usepackage{amssymb,stmaryrd}
\usepackage{amsfonts,url,comment}
\usepackage[linesnumbered,ruled,vlined]{algorithm2e}
\usepackage{graphicx,color}
\usepackage{subfigure}
\usepackage[linesnumbered,ruled,vlined]{algorithm2e}
\newcommand{\Trop}{\text{trop\,}}
\newcommand{\tconv}{\text{tconv\,}}

\newtheorem{thm}{Theorem}

\newtheorem{prop}[thm]{Proposition}
\newtheorem{conj}[thm]{Conjecture}

\newcommand{\RR}{\mathbb{R}}
\newcommand{\R}{\mathbb{R}}

\DeclareMathOperator*{\argmax}{arg\,max}
\begin{document}

\title{Tree Topologies along a Tropical Line Segment\thanks{R.Y.~is partially supported by NSF (DMS 1916037)}
}


\author{Ruriko Yoshida  \and Shelby Cox   
}


\institute{R. Yoshida \at
            Naval Postgraduate School\\
              1411 Cunningham Road\\
Naval Postgraduate School\\
Monterey, CA 93943-5219\\
              Tel.: +1-831-656-2973\\
              Fax: +1-831-656-2595\\
              \email{ryoshida@nps.edu}           
           \and
           S. Cox \at
            University of Michigan\\
           2074 East Hall\\
            530 Church Street\\
            Ann Arbor, MI 48109-1043 \\
}

\date{Received: date / Accepted: date}

\maketitle

\begin{abstract}
Tropical geometry with the max-plus algebra has been applied to statistical learning models over tree spaces because geometry with the tropical metric over tree spaces has some nice properties such as convexity in terms of the tropical metric.  One of the challenges in applications of tropical geometry to tree spaces is the difficulty interpreting outcomes of statistical models with the tropical metric.  This paper focuses on combinatorics of tree topologies along a tropical line segment, an intrinsic geodesic with the tropical metric, between two phylogenetic trees over the tree space and we show some properties of a tropical line segment between two trees.  Specifically we show that a probability of a tropical line segment of two randomly chosen trees going through the origin (the star tree) is zero if the number of leave is greater than four, and we also show that if two given trees differ only one nearest neighbor interchange (NNI) move, then the tree topology of a tree in the tropical line segment between them is the same tree topology of one of these given two trees with  possible zero branch lengths. 
\end{abstract}

\section{Introduction}

Due to the increasing amount of data today, data science is one of the most exciting fields in science.  It finds applications in statistics, computer science, business, biology, data security, physics, and so on. 
Most statistical models in data sciences assume that data points in an input sample are distributed over a Euclidean space if they have numerical measurements.  However, in some cases this assumption can fail.  For example, a {\em space of phylogenetic trees} with a fixed set of leaves is a union of lower dimensional cones over $\mathbb{R}^e$, where $e = \binom{n}{2}$ with $n$ as the number of leaves \cite{AK}.  Since the space of phylogenetic trees is a union of lower dimensional cones, we cannot just apply statistical models in data science to a set of phylogenetic trees \cite{YZZ}.  

There has been much work in spaces of phylogenetic trees. In 2001, Billera-Holmes-Vogtman (BHV) developed the notion of a space of phylogenetic trees with a fixed set of labels for leaves \cite{BHV}, which is a set of all possible unrooted phylogenetic trees with the fixed set of labels on leaves and which is a union of orthants; each orthant contains possible unrooted phylogenetic trees with a fixed tree topology.  
They also showed that this space is ${\rm CAT}(0)$ space so that there is a unique 
shortest connecting path, or geodesic, between any two points in the space defined by the ${\rm CAT}(0)$-metric.   We can also generalize this tree space to the space of rooted phylogenetic trees with a given set of leaves.  

There is some work in development on machine learning models with the BHV metric.  For example, Nye defined the notion of the first order
principal component geodesic as the unique geodesic with the BHV metric over the tree space which minimizes the sum of residuals between the geodesic and each data point \cite{Nye}.  However, we cannot use a convex hull under the BHV metric for higher principal components because Lin et al.~showed that the convex hull of three points with the BHV metric over the tree space can have arbitrarily high dimension \cite{LSTY}.

Another space of phylogenetic trees with a given set of leaves is the edge-product space \cite{10.2307/23238597,10.1093/sysbio/syx080}.  Metrics defined over the edge-product space are associated with probability distributions on characters and the Hellinger and Jensen–Shannon metrics are used between two distributions over the edge-product space with a given set of leaves \cite{10.1093/sysbio/syx080}.  This space is also well studied from the view of algebraic geometry (for example, \cite{10.2307/23238597}). Since the edge-product space is based on distributions on a set of characters to represent phylogenetic trees with a given set of leaves, it is natural to conduct statistical analysis over such tree spaces using information geometry.  For more details,  Garba et al.~summarize these three tree spaces in a recent their work \cite{Nye2021} and they extended the edge-product space to a new tree space called the {\em Wald space}. 

In 2004, Speyer and Sturmfels showed a space of phylogenetic trees with a given set of labels on their leaves is a tropical Grassmannian \cite{SS}, which is a tropicalization of a linear space defined by a set of linear equations \cite{YZZ} with the max-plus algebra.  It is important to note that the tree space defined by Speyer and Sturmfels is not isometric to the tree space defined by Billera-Holmes-Vogtman although they are homeomorphic to each other. The first attempt to apply tropical geometry to computational biology and statistical models was done by Pachter and Sturmfels \cite{PS}. 
The tropical metric with the max-plus algebra on the tree space is known to behave very well \cite{AGNS,CGQ}.  For example, contrary to the BHV metric, the dimension of the convex hull of $s$ tropical points is at most $s-1$ \cite{LSTY}.  
There has been much work done with the tropical metric over the tree space of equidistant trees to analyze a set of phylogenetic trees with $n$ leaves $X = \{1, 2, \ldots , n\}$.  For example, Yoshida et al.~defined tropical principal component analysis (PCA) with the tropical metric over the space of {\em equidistant trees} to reduce dimensionality and to visualize data sets \cite{YZZ}.  Also Tang et al.~developed hard and soft tropical support vector machines (SVMs) and the authors applied them to classifying sets of equidistant trees \cite{TWY}.  For more details on applications of tropical geometry to tree spaces, see \cite{Yoshida2}.
  

One of the challenges in statistical learning models with the tropical metric over the space of equidistant trees is the difficulty to interpret outputs from such methods.  For example, the principal geodesic developed by Nye in \cite{Nye} has a natural interpretation of the geodesic with the BHV metric over the space of phylogenetic trees with $n$ leaves.  However, it is not obvious how to interpret a tropical principal polytope developed by Yoshida et al.~in \cite{YZZ,10.1093/bioinformatics/btaa564}. Interpretation of the output from a statistical learning model is one of the most important processes in data analysis.  Therefore, this paper focuses on the interpretation of a tropical ``geodesic'' with the tropical metric on the space of equidistant trees with $n$ leaves.  However, a tropical geodesic between two equidistant trees is known to be not unique (e.g., see \cite{anthea}). In fact, there are infinitely many tropical geodesics between two points.  This makes it difficult to analyze behavior of a tropical geodesic between trees, thus in this paper we consider a {\em tropical line segment} between trees which is intrinsic and unique on the space of equidistant trees \cite{anthea}. 
Thus, here we use a tropical line segment between two equidistant trees as a tropical geodesic between them. 

In this paper, we focus on rooted phylogenetic trees with $n$ leaves.  More specifically we focus on {\em equidistant trees}, rooted phylogenetic trees whose total branch lengths from the root to each leaf are the same for all leaves.  It is important to note that the tree spaces defined by Billera-Holmes-Vogtmann \cite{BHV},  Speyer-Sturmfels \cite{SS}, and the edge-product space with a distribution based metric \cite{10.1093/sysbio/syx080} can be applied to a space of rooted phylogenetic trees, but they do not assume equidistant trees.

Among these three tree spaces: the BHV space; the edge-product space; and the tree space with the tropical metric, a tree space with the tropical metric has the least attention because the geodesic between trees with the tropical metric is not unique and also is hard to  interpret in terms of tree topologies. Therefore we focus on combinatrics of tree topologies on the geodesic between trees with the tropical metric, especially {\em tropical line segment} between trees, which is a unique geodesic in terms of the tropical metric. 
 Monod et al. investigated tree topologies along a tropical line segment over the space of ultrametrics and characterizes symmetry of tree topologies on a tropical line segment in \cite{LMY2}. 
Therefore, we investigate explicitly how tree topologies change over a tropical line segment for some specific cases.  Specifically, we show that the probability of a tropical line segment of two randomly chosen trees going through the origin (the star tree) is zero if the number of leaves is greater than four.  In addition, we also show that if two given trees differ by only one nearest neighbor interchange (NNI) move, then the tree topology of a tree in the tropical line segment between them is the same tree topology of one of these two given trees with  possible zero branch lengths.
We end this paper with a conjecture that tree topologies of trees on a tropical line segment change by a sequence of NNI moves. 
Through this paper we propose open problems to understand combinatorics of tree topologies along a tropical line segment between equidistant trees. 

\section{Notation and Definitions}

\subsection{Tropical Basics}

In the tropical semiring $(\,\R \cup \{-\infty\},\oplus,\odot)\,$, we define the basic
operations of addition and multiplication as:
$$a \oplus b := \max\{a, b\}, ~~~~ a \odot b := a + b ~~~~\text{  where } a, b \in \R.$$
In this semiring, the identity element for addition is $-\infty$ and $0$ is the identity element for multiplication.  An essential feature of tropical arithmetic is that there is no subtraction. 
Tropical division is defined to be classical subtraction, so $(\,\R \cup \{-\infty\},\oplus,\odot)\,$ satisfies all ring axioms (and indeed field axioms) except for the existence of an additive inverse. In tropical geometry, we work on the \emph{tropical projective torus}, $\mathbb{R}^e \!/\mathbb R {\bf 1}$, where ${\bf 1}$ denotes the all-ones vector, i.e., for any $x=(x_1, \ldots , x_e) \in \mathbb{R}^e \!/\mathbb R {\bf 1},$
\[
(x_1, \ldots , x_e) = (x_1 + c, \ldots , x_e + c) ,
\] for any constant $c \in \RR$. 

\begin{definition}
Over the tropical semiring $(\,\R \cup \{-\infty\},\oplus,\odot)\,$, suppose $u = (u_1, \ldots , u_e),\, v = (v_1, \ldots , v_e)$ are in the tropical projective space $\mathbb{R}^e / \mathbb{R} {\bf 1}$. Then the tropical metric $d_{\rm tr}$ is defined as
\[
d_{\rm tr}(u, v) = \max_i\{u_i - v_i\} - \min_i\{u_i - v_i\}.
\]
\end{definition}

\begin{definition}[Tropical Convex Hull]\label{def:polytope}
The {\em tropical convex hull} or {\em tropical polytope} of a given finite subset $V = \{v^1, \ldots , v^s\}\subset \mathbb R^e \!/\mathbb R {\bf   1}$ is the smallest tropically-convex subset containing $V \subset \mathbb R^e \!/\mathbb R {\bf  1}$: it is written as the set of all tropical linear combinations of $V$ such that:
  $$ \mathrm{tconv}(V) = \{a_1 \odot v^1 \oplus a_2 \odot v^2 \oplus \cdots \oplus a_s \odot v^s \mid  a_1,\ldots,a_s \in \R \}.
$$
A {\em tropical line segment} between two points $v^1, \, v^2$ is the tropical polytope of $\{v^1, \, v^2\}$.
\end{definition}
\begin{example}\label{ex:lineseg1}
Suppose we have two vectors 
\[
v^1 = (0, 0, 0), \, v^2 = (0, 3, 1),
\]
over $\mathbb R^3 \!/\mathbb R {\bf 1}$. Then we consider the tropical line segment between $v^1$ and $v^2$, that is
$$ \mathrm{tconv}(V) = \{a_1 \odot v^1 \oplus a_2 \odot v^2  \mid  a_1, a_2 \in \R \}.
$$
Note that 
\[
a_1 \odot v^1 = (a_1 + 0, a_1 + 0, a_1 + 0) = (0, 0, 0) = v^1,
\]
and 
\[
a_2 \odot v^2 = (a_2 + 0, a_2 + 3, a_2 + 1) = (0, 3, 1) = v^2.
\]
Also 
\[
\begin{array}{rl}
& a_1 \odot v^1 \oplus a_2 \odot v^2  \\
= & (\max\{a_1 + 0, a_2 + 0\}, \max\{a_1 + 0, a_2 + 3\}, \max\{a_1 + 0, a_2 + 1\})\\
=& (\max\{a_1, a_2\}, \max\{a_1, a_2 + 3\}, \max\{a_1, a_2 + 1\}).
\end{array}
\]

When we have $a_1 \geq a_2 + 3$, then 
\[
\begin{array}{rl}
& a_1 \odot v^1 \oplus a_2 \odot v^2 \\
=& (\max\{a_1, a_2\}, \max\{a_1, a_2 + 3\}, \max\{a_1, a_2 + 1\})\\
=& (a_1, a_1, a_1)\\
=& v^1 .
\end{array}
\]

When we have $a_2+1 \leq a_1 < a_2+3$, then
\[
\begin{array}{rl}
& a_1 \odot v^1 \oplus a_2 \odot v^2  \\
=& (\max\{a_1, a_2\}, \max\{a_1, a_2 + 3\}, \max\{a_1, a_2 + 1\})\\
=& (a_1, a_2+3, a_1)\\
=& (0, a_2+3 - a_1, 0) \\
=& (0, l, 0),\\
\end{array}
\]
where $l = a_2+3 - a_1$ for $0 < l < 2$.

When we have $a_2 \leq a_1 < a_2+1$, then
\[
\begin{array}{rl}
& a_1 \odot v^1 \oplus a_2 \odot v^2  \\
=& (\max\{a_1, a_2\}, \max\{a_1, a_2 + 3\}, \max\{a_1, a_2 + 1\})\\
=& (a_1, a_2+3, a_2 + 1)\\
=& (0, a_2+3 - a_1, a_2 + 1 - a_1) \\
=& (0, l + 2, l),\\
\end{array}
\]
where $l = a_2+3 - a_1 + 1$ for $0 < l < 1$.
Therefore the tropical line segment $v^1$ and $v^2$ are the line segments from $v^1$ to $(0, 2, 0)$ and from $(0, 2, 0)$ to $v^2$.
\end{example}

Let $\Gamma_{u, v}$ be a tropical line segment between two ultrametrics $u, \, v \in \mathcal{U}_n$. 

\begin{example}
Suppose we have a set $ V = \left\{v^1, \, v^2, \, v^3\right\} \subset \mathbb R^3 \!/\mathbb R {\bf 1}$ where
\[
v^1 = (0, 0, 0), \, v^2 = (0, 3, 1), \, v^3 = (0, 2, 5).
\]

First, we compute the tropical line segment $\Gamma_{v^1, v^2}$ as Example \ref{ex:lineseg1} shows.  Then, 
similarly we can compute the tropical line segment $\Gamma_{v^2, v^3}$ which is a line segment of
\[
v^2 = (0, 3, 1), \, (0, 3, 5), \, v^3 = (0, 2, 5),
\]
and the tropical line segment $\Gamma_{v^1, v^3}$ is a line segment of
\[
v^1 = (0, 0, 0), \, (0, 0, 3), \, v^3 = (0, 2, 5).
\]
The tropical polytope $\mathrm{tconv}(V)$ of $V$ is shown in Fig. \ref{fig:tropPoly}.
 \begin{figure}
 \centering
\includegraphics[width=0.5\textwidth]{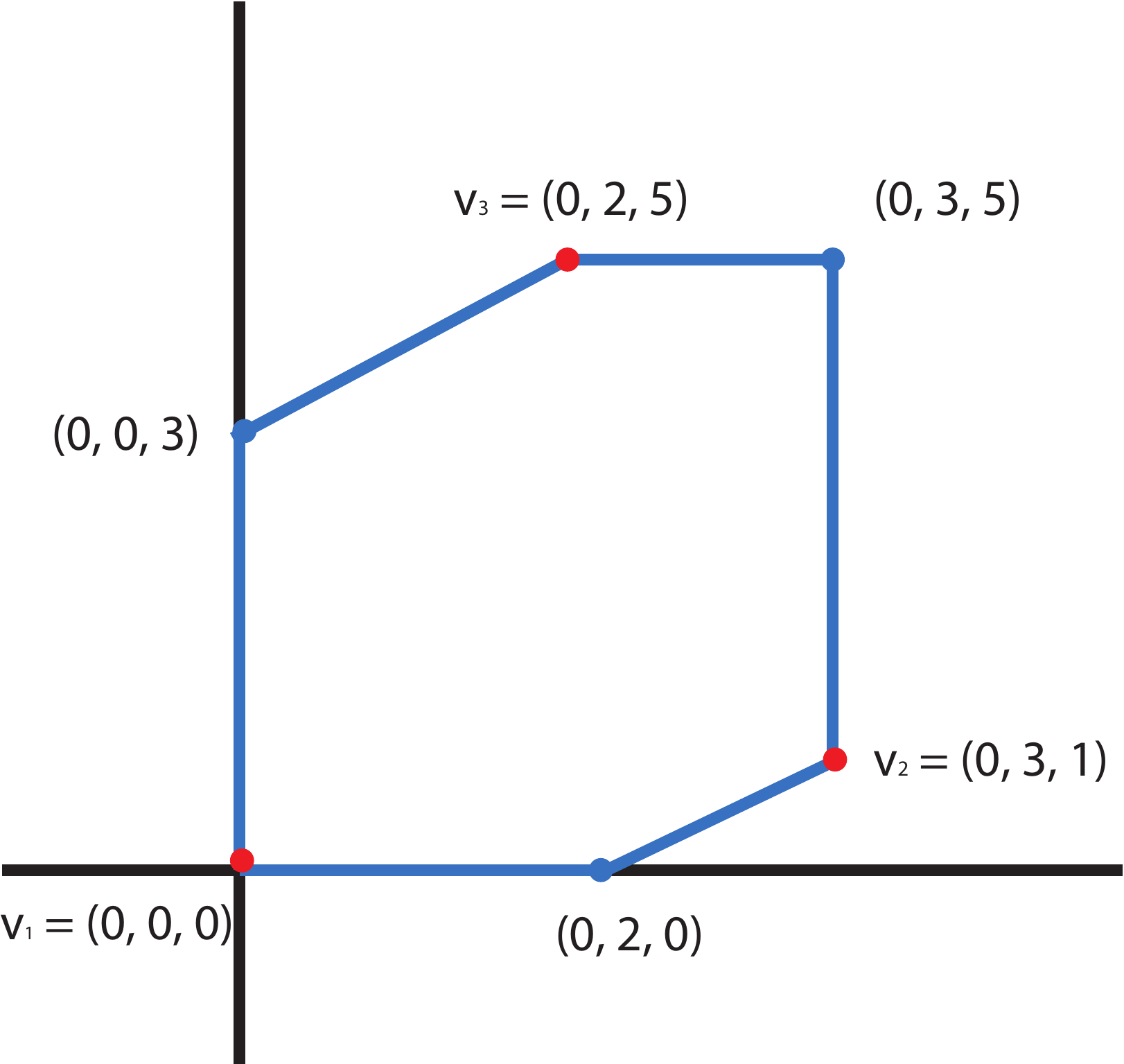}
\caption{Tropical polytope of three points $(0, 0, 0), \, (0, 3, 1), \, (0, 2, 5)$ in $\mathbb R^3 \!/\mathbb R {\bf 1}$.}
\label{fig:tropPoly}
\end{figure}
\end{example}

\subsection{Space of Ultrametrics}

A phylogenetic tree is a weighted tree with labels $X = \{1, \ldots , n\}$ on its leaves and its internal nodes do not have any labels.  Each edge on a phylogenetic tree has non-negative weight which represents evolutionary time and mutation rates. A phylogenetic tree can be rooted or unrooted.  Throughout this paper we assume that all phylogenetic trees are rooted.  Also we assume that all phylogenetic trees are equidistant trees, that is, rooted phylogenetic trees with the property that the distance from the root to each leaf is the same for all leaves and all trees have the same height.  This is the same assumption of the multispecies coalescent model \cite{coalescent}, one of the most popular models to model gene trees under the species tree.

A dissimilarity map is a map $u_{ij}: X \times X \to \mathbb{R}_+$ such that
\[
u_{ij} = \begin{cases}
u_{ji} & \mbox{if } i \not = j\\
0 & \mbox{if } i = j,
\end{cases}
\]
for any $i, j \in X$. 
Suppose $u = (u_{12}, \ldots , u_{n-1 n})$ are dissimilarity maps.  Then, note that $u$ is a {\em metric} if $u$ satisfies the triangle inequality.  If there exists a phylogenetic tree $T$ with the leaf label $X$ such that $u_{ij}$ is a pairwise distance from a leaf $i \in X$ to a leaf $j \in X$, then we call $u$ a {\em tree metric}.


\begin{definition}
$u := (u_{12}, u_{13}, \ldots , u_{n-1 n})$ is called an {\em ultrametric}  if 
\[
\max \{u_{ij}, \, u_{ik}, \, u_{jk}\}
\]
for distinct $i, \, j, \, k \in X= \{1, 2, \ldots n\}$ is achieved at least twice.  
\end{definition}

It is well-known that 
$u = (u_{12}, u_{13}, \ldots , u_{n-1 n})$ is an ultrametric if and only if $u = (u_{12}, u_{13}, \ldots , u_{n-1 n})$ is a tree metric with an equidistant tree $T$ \cite{anthea}.  In phylogenetics it is called {\em the three point condition} \cite{buneman1974note}. Therefore, here we work on the space of ultrametrics as a space of equidistant trees with $n$ leaves.  Let $\mathcal{U}_n$ denote the space of all ultrametrics in the tropical projective space $\mathbb R^e/{\bf 1}\mathbb R$ where $e = \binom{n}{2}$.

Note that $\mathcal{U}_n$ can be considered as the space of equidistant trees since for each equidistant tree there is a unique ultrametric to define the equidistant tree.  

Let $L_n$ be the subspace of $\mathbb R^e$ defined by the linear equations $x_{ij} - x_{ik} + x_{jk}=0$ for $1\leq i < j <k \leq n$, where $x_{ij}, x_{ik}, x_{jk}$ are variables. The tropicalization $\Trop(L_n)\subseteq \RR^e/\RR {\bf 1}$ is the tropical linear space consisting of points $(v_{12},v_{13},\ldots, v_{n-1,n})$ such that $\max(v_{ij},v_{ik},v_{jk})$ is obtained at least twice for all triples $i,j,k\in \{1, \ldots , n\}$.  Then we have the following theorem.  
\begin{thm}[Theorem 3 in \cite{YZZ}]
\label{ultrametrics}
The image of $\mathcal{U}_n$ in $\RR^e/\RR {\bf 1}$ is $\Trop(L_n)$.
\end{thm} 
Therefore, we can think of $\mathcal{U}_n$  as a tropical linear space.  
Also note that $\mathcal{U}_n$ is a tropical linear space over the tropical projective space.  Therefore $\mathcal{U}_n$ is {\rm tropically convex}.  Thus, if we take any two points $u, \, v \in \mathcal{U}_n$ then the tropical intrinsic geodesic between $u$ and $v$ is in $\mathcal{U}_n$.  This means that all points in a tropical intrinsic geodesic between $u, \, v \in \mathcal{U}_n$ are ultrametrics and there are equidistant trees associated with these ultrametrics.  
These leads to the following problem:
Suppose we have two equidistant trees $T_u, \, T_v$ with their ultrametrics $u, \, v \in \mathcal{U}_n$, respectively.  Then how do tree topologies change along a tropical intrinsic geodesic between ultrametrics $u, \, v$?

This is one of the most important questions in order to develop data science models using the tropical metric over $\mathcal{U}_n$ since this will answer the {\em interpretation} of results from a model with the tropical metric.  For example, the output from the tropical principal component analysis (PCA) is not obvious to interpret. 

It is very important to note that tropical geodesics between two points are {\em not} unique.   For example, we consider two points from Example \ref{ex:lineseg1}.  Then we have the distance over the tropical line segment between $v_1 = (0, 0, 0)$ and $v_2 = (0, 3, 1)$ is $d_{\rm tr}(v_1, (0, 2, 0)) + d_{\rm tr}((0, 2, 0), v_2) = 2 + 1 = 3$.  But the straight line from $v_1$ to $v_2$ is also a tropical geodesic since $d_{\rm tr}(v_1, v_2) = 3$.  

However, if we consider a {\em tropical line segment} defined in Definition \ref{def:polytope} between two ultrametrics $u, \, v \in \mathcal{U}_n$ associated with equidistant trees, that is, a {\em tropical polytope} generated by $u, \, v \in \mathcal{U}_n$, then the tropical line segment is unique.  
Let $\Gamma_{u, v}$ be a tropical line segment between two ultrametrics $u, \, v \in \mathcal{U}_n$. 
\begin{prop}
A tropical line segment  $\Gamma_{u, v}$ of $u, \, v \in \mathbb{R}^e \!/\mathbb R {\bf 1}$ is unique, it is a geodesic  in $\mathcal{U}_n$, and it is intrinsic.  
\end{prop}
\begin{proof}
The first part of the statement is directly from Proposition 24 in \cite{anthea}.  The second and the third statements can be proven by the fact that $\mathcal{U}_n$ is a tropical linear space by Theorem \ref{ultrametrics} and a tropical line segment is a tropical polytope of two points in $\mathcal{U}_n$.
\end{proof}

Thus, we consider the following question:
\begin{problem}\label{prob1}
Suppose we have two equidistant trees $T_u, \, T_v$ with their ultrametrics $u, \, v \in \mathcal{U}_n$, respectively.  Then how do tree topologies change along the tropical line segment between ultrametrics $u, \, v$?
\end{problem}

We can generalize this question to a tropical polytope generated by finitely many ultrametrics $u^1, \, \ldots ,  u^k \in \mathcal{U}_n$.
\begin{problem}
Suppose we have $k$ equidistant trees $T_1, \, \ldots , T_k$ with their ultrametrics $u^1, \, \ldots , u^k \in \mathcal{U}_n$, respectively.  Then how do tree topologies change in the tropical polytope generated by ultrametrics $u^1,  \, \ldots ,  u^k$?
\end{problem}

In the paper by Page et al.~in \cite{10.1093/bioinformatics/btaa564}, we partially addressed this problem.  
\begin{definition}
\label{def:types}
Let $\mathcal P = \tconv(u^{1}, \dots, u^{k})\subseteq \mathbb R^{e}/\RR{\bf 1}$ be a tropical polytope. Each point $x = (x_1, \ldots , x_e)$ in $\mathbb R^{ e}/\RR{\bf 1}$ has a \emph{type} ${Q} = ({Q}_1,\dots, {Q}_{ e })$ according to $\mathcal P$, where an index $i$ is in ${Q}_j$ if
\[u^{i}_j - x_j = \max(u^{i}_1-x_1,\dots, u^{i}_{ e }-x_{ e }),\]
where $u^{i} = (u^i_1, \ldots , u^i_e)$ for $i = 1, \ldots , e$ and $j = 1, \ldots , e$.
The tropical polytope $\mathcal P$ consists of all points $x$ whose type ${Q} = ({Q}_1,\dots, {Q}_{e})$ has all ${Q}_i$ nonempty. Each collection of points with the same type is called a \emph{cell}.
\end{definition}
For more details on a cell of a tropical polytope, see \cite{10.1093/bioinformatics/btaa564}.

\begin{thm}[\cite{10.1093/bioinformatics/btaa564}]\label{thm:topology}
Let $\mathcal P = \tconv(u^{1}, \dots, u^{k})\subseteq \mathbb R^{n}/\mathbb{R}{\bf 1}$ be a tropical polytope spanned by ultrametrics. Then any two points $x$ and $y$ in the same cell of $\mathcal P$ are also ultrametrics with the same tree topology.
\end{thm}

Now it is natural to ask the following question:
\begin{problem}
How do tree topologies change if a tropical geodesic crosses between two cells on a tropical polytope $\mathcal P$ over $\mathcal{U}_n$?
\end{problem}

\section{Drawing a Tropical Line Segment on $\mathcal{U}_n$}

In this section we interpret the algorithm to compute a tropical line segment in \cite{MS}.  
In order to compute the tropical line segment between equidistant trees $T_1$ and $T_2$ with leaves $X$, we adapt the algorithm shown in the proof of Proposition 5.2.5 in \cite{MS}.  Note that Proposition 5.2.5 uses the min-plus algebra, whereas we are using the max-plus algebra.  Recall that we use the max-plus algebra because the tree space is a tropical Grassmannian with the max-plus algebra \cite{SS}.

Let $\Gamma_{u, v}$ be a tropical line segment  between two ultrametrics $u, \, v \in \mathcal{U}_n$.  Suppose $u$ and $v$ are ultrametrics corresponding to equidistant trees $T_u$ and $T_v$, then we use a notation $\Gamma_{T_u, T_v}$ as well.  
First we adapt the algorithm shown in the proof of Proposition 5.2.5 in \cite{MS}.

\begin{algorithm}[H]
\DontPrintSemicolon 
\KwIn{A point $u=(u_{1}, \ldots , u_{e})$ and a point $v=(v_{1}, \ldots , v_{e})$ in $\RR^e /\RR{\bf 1}$}
\KwOut{A tropical line segment  $\Gamma_{u, v}$ between $u$ and $w$ in $\RR^e /\RR{\bf 1}$}
Compute $\lambda = v - u = (v_{1} - u_{1}, \ldots , v_{e} - u_{e})$.\;
Set $L = \emptyset$ and set $y^0 = v$.\;
\For{$i \gets 1$ \textbf{to} $e$}{
    Find the $i$th smallest coordinate in $\lambda$, and denote its value as $\lambda_{i}$.\;
    Set $y^i = \left(\max\{\lambda_{i}+u_{1}, v_{1}\}, \ldots ,  \max\{\lambda_{i}+u_{e}, v_{e}\}\right)$.\;
    Set $L = L \cup \{y^i\}$.\;
}
Set $y^{e+1} = u$.\;
\Return{the line segments of lines from $y^i$ to $y^{i+1}$ in $L$ for $i = 0, 1, \ldots , e$.}\;
\caption{Tropical line segment  in $\RR^e /\RR{\bf 1}$}\label{alg:trop_line0}
\end{algorithm}

Now the following algorithm is to compute a tropical line segment between ultrametrics in $\mathcal{U}_n$ modified from Algorithm \ref{alg:trop_line0}:

\begin{algorithm}[H]
\DontPrintSemicolon 
\KwIn{An ultrametric $u=(u_{12}, \ldots , u_{n-1 n})$ computed from an equidistant tree $T_1$ with $n$ leaves and an ultrametric $v=(v_{12}, \ldots , v_{n-1n})$ computed from an equidistant tree $T_2$ with $n$ leaves}
\KwOut{A tropical line segment  $\Gamma_{u, v}$ between $u$ and $w$ in $\mathcal{U}_n$}
Compute $\lambda = v - u = (v_{12} - u_{12}, \ldots , v_{n-1n} - u_{n-1 n})$.\;
Set $L = \emptyset$ and set $y^0 = v$.\;
\For{$i \gets 1$ \textbf{to} ${n \choose 2}$}{
    Find the $i$th smallest coordinate $(i_1, i_2)$ in $\lambda$, and denote its value as $\lambda_{i_1 i_2}$.\;
    Set $y^i = \left(\max\{\lambda_{i_1 i_2}+u_{12}, v_{12}\}, \ldots ,  \max\{\lambda_{i_1 i_2}+u_{n-1 n}, v_{n-1 n}\}\right)$.\;
    Set $L = L \cup \{y^i\}$.\;
}
Set $y^{\binom{n}{2}+1} = u$.\;
\Return{the line segments of lines from $y^i$ to $y^{i+1}$ in $L$ for $i = 0, 1, \ldots , \binom{n}{2}$.}\;
\caption{Tropical line segment  in $\mathcal{U}_n$}\label{alg:trop_line1}
\end{algorithm}

\begin{example}
Suppose we have two points 
\[
v^1 = (0, 0, 0), \, v^2 = (0, 3, 1)
\]
from Example \ref{ex:lineseg1}.
We wish to compute the tropical line segment $\Gamma_{v^1, v^2}$ using Algorithm \ref{alg:trop_line0}.  

First we compute $\lambda = v^2 - v^1 = (0, 3, 1)$.  Then, order elements of $\lambda$ from the smallest to the largest, that is
\[
\lambda = (0, 1, 3).
\]
Then first we compute
\[
\max(\lambda_1 + v^1, v^2) = (\max\{0+0, 0\}, \max\{0+0, 3\}, \max\{0+0, 1\}) = (0, 3, 1).
\]
This is one of the end points of $\Gamma_{v^1, v^2}$.
Then we compute 
\[
\max(\lambda_2 + v^1, v^2) = (\max\{1+0, 0\}, \max\{1+0, 3\}, \max\{1+0, 1\}) = (1, 3, 1) = (0, 2, 0).
\]
This is a point where $\Gamma_{v^1, v^2}$ bends.
Finally, we compute 
\[
\max(\lambda_3 + v^1, v^2) = (\max\{3+0, 0\}, \max\{3+0, 3\}, \max\{3+0, 1\}) = (3, 3, 3) = (0, 0, 0).
\]
This is one of the end points of $\Gamma_{v^1, v^2}$.
Thus,  $\Gamma_{v^1, v^2}$ is a line segment 
\[
v^1 = (0, 0, 0), \, (0, 2, 0), \, v^2 = (0, 3, 1).
\]
\end{example}

Here we show some properties of tropical line segments:
\begin{prop}\label{same_tree_topology}
Let  $\Gamma_{u, v}$ be the tropical line segment between ultrametrics $u, v \in \mathcal{U}_n$.  If $x, y$ are in one straight line on $\Gamma_{u, v}$, then $x$ and $y$ have the same tree topology.
\end{prop}

\begin{proof}
This is a corollary of Theorem \ref{thm:topology}.
\end{proof}

\begin{thm}\label{th:uniform}
If we take $x, y$ randomly from $\mathcal{U}_n$ from a uniform distribution on equidistant trees with fixed height of the equidistant trees, then the tropical line segment  $\Gamma_{x, y}$ between $x, y \in  \mathcal{U}_n$ goes through the star tree (the origin in terms of ultrametrics) with probability zero for $n \geq 5$.  
\end{thm}

\begin{proof}
Note that the tropical line segment $\Gamma_{x, y} \subset \mathcal{U}_n$ between $x, y \in \mathcal{U}_n$ is a tropical polytope generated by $x, y$.  We will use Lemma 3.3 from \cite{10.1093/bioinformatics/btaa564} which states that the origin is contained in $\Gamma_{x, y} $ if and only if $x\oplus y=0.$
Suppose the trees both have height $h > 0$.

When $n = 3$, if two trees are in the different polyhedral cones, the tropical line segment between them has to go through the origin. 

When $n = 4$, suppose we have two trees with ultrametrics such that
\[
\begin{array}{c}
   x =  (2a_1, 2h, 2h, 2h, 2h, 2h,)  \\
   y =  (2h, 2a_2, 2h, 2h, 2b_2, 2h)
\end{array}
\]
for any $a_1, a_2, b_1, b_2 < h$, then we have $x\oplus y=0.$  So the tropical line segment between them contains the origin for any $a_1, a_2, b_1, b_2 < h$.  

For $n \geq 5$,  consider the two trees $T_1$ and $T_2$ shown in Figure \ref{fig:tree:uniform}. $R_1, R_2, L_1, L_2 \subset [n]:=\{1, \ldots n\}$ such that $L_1, R_1$ partition $[n]$, i.e., $L_1 \cap R_1 = \emptyset$ and $L_1 \cup R_1 = [n]$, and  $L_2, R_2$ partition $[n]$, i.e., $L_2 \cap R_2 = \emptyset$ and $L_2 \cup R_2 = [n]$. Denote $L \sqcup R = [n]$ be a partition $\{R, \, L\}$ of $[n]$ for $R, L \subset [n]$. Suppose $x$ is the ultrametric of $T_1$ and $y$ is the ultrametric of $T_2$.
\begin{figure}
    \centering
    \includegraphics[width=\textwidth]{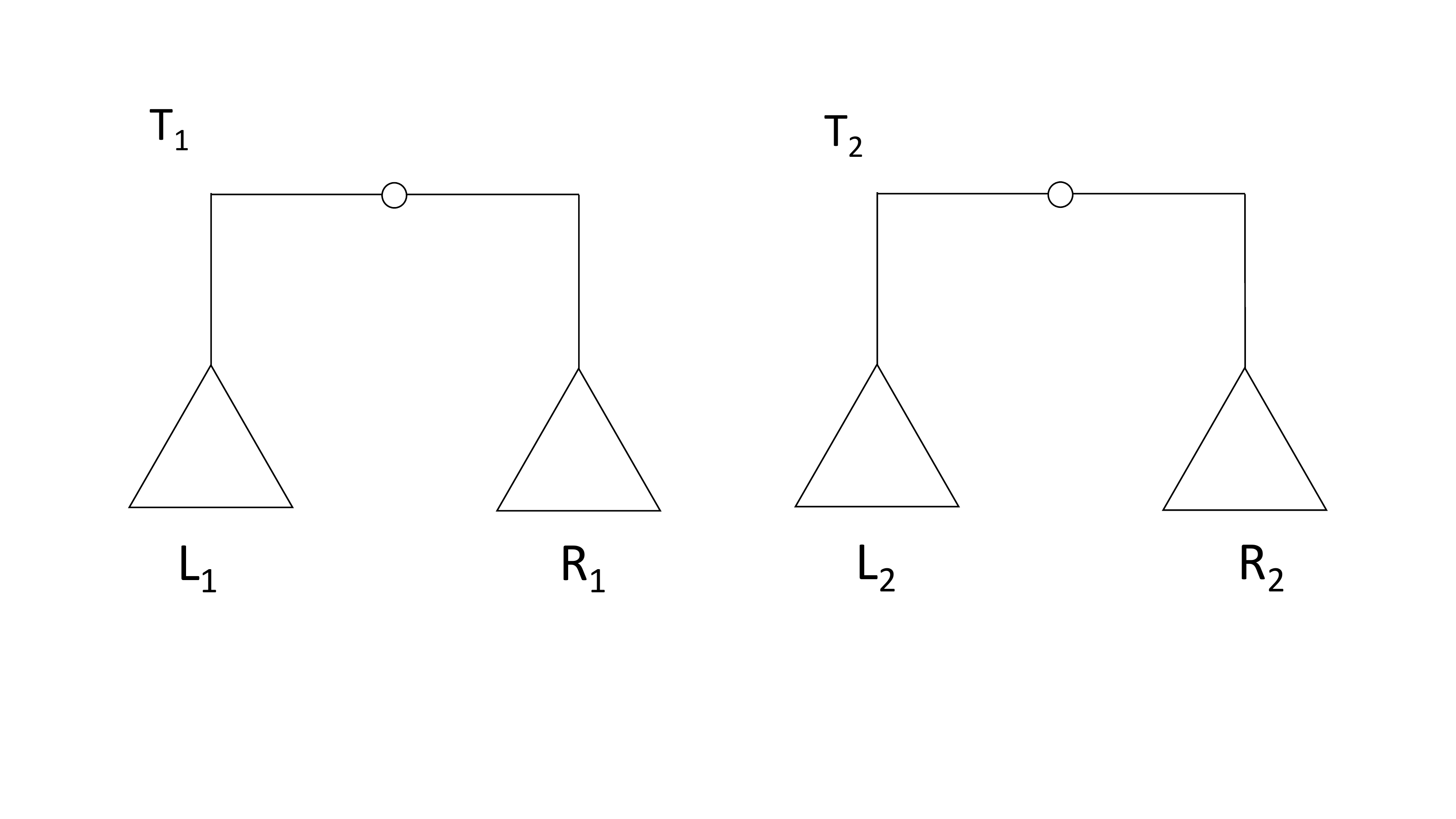}
    \caption{Two trees the proof of Theorem \ref{th:uniform} for $n \geq 5$.}
    \label{fig:tree:uniform}
\end{figure}

Also note that
\[
(L_1 \cap L_2) \sqcup (L_1 \cap R_2) \sqcup (L_2 \cap R_1) \sqcup (L_2 \cap R_2) = [n].
\]
Since $n \geq 5$, therefore, one of the four sets must have at least two elements. Since $n \geq 5$, there exist a pair of two leaves $i, j$ from one of $(L_1 \cap L_2)$, $(L_1 \cap R_2)$, $(L_2 \cap R_1)$, and $(L_2 \cap R_2)$.  Then we have
\[
x_{ij} < 2h, \mbox{ and } y_{ij} < 2h.
\]
Therefore, $(x\oplus y)_{ij}< 2h$.  Therefore, $x\oplus y$ is not the star tree, so by Lemma 3.3 from \cite{10.1093/bioinformatics/btaa564}, the tropical line from $x$ to $y$ does not pass through the origin, i.e., the star tree.
\end{proof}


Now we interpret Algorithm \ref{alg:trop_line1} in terms of equidistant trees.  Before the algorithm to draw a tropical line between two equidistant trees, we have the following definition:
\begin{definition}
Suppose we have an equidistant tree $T$ with $n$ leaves.  An {\em external branch} of $T$ is an edge directly attached to a leaf $i \in X$.
\end{definition}
\begin{definition}
Let $T$ be an equidistant tree with $n$ leaves and let $b_i$ be the external branch lengths of leaf $i$ in $T$  for $i = 1, \ldots , n$. Also let $\alpha T$ be an equidistant tree with $n$ leaves such that the tree topology of $\alpha T$ is the same as the tree topology of $T$ and its external branch length of leaf $i$ is $\alpha/2 + b_i$ for $i = 1, \ldots , n$.
\end{definition}

\begin{algorithm}[H]
\DontPrintSemicolon 
\KwIn{An ultrametric $u=(u_{12}, \ldots , u_{n-1 n})$ computed from an equidistant tree $T_1$ with $n$ leaves and an ultrametric $v=(v_{12}, \ldots , v_{n-1n})$ computed from an equidistant tree $T_2$ with $n$ leaves}
\KwOut{A tropical line segment  $\Gamma_{u, v}$ between $u$ and $v$ in $\mathcal{U}_n$ written a sequence of equidistant trees associated with segments in the tropical line segments}
Compute $\lambda = (\lambda_{12}, \ldots , \lambda_{n-1 n}) = v - u = ((v_{12} - u_{12}), \ldots , (v_{n-1n} - u_{n-1 n}))$.\;
Set $L = \emptyset$ and $T^0 = T_2$.
\For{$i \gets 1$ \textbf{to} ${n \choose 2}$}{
    Find the $i$th smallest coordinate $(i_1, i_2)$ in $(\lambda_{12}, \ldots , \lambda_{n-1 n})$, and denote its value as $\lambda_{i_1 i_2}$.\;
    Initialized $T^i$ as a star tree.\;
    \For{each pair of leaves $(k, \, l)$}{
        Compare the height of the internal node of $(k, \, l)$ in $\lambda_{i_1i_2} T_1$ which equals to $\lambda_{i_1 i_2} + u_{kl}$ with the height of the internal node of $(k, \, l)$ in $T_2$, which equals to $v_{kl}$. \;
        Set an internal node with its height $\max\{\lambda_{i_1 i_2} + u_{kl}, v_{kl}\}$ which is the ancestor of leaves $(k, \, l)$.\;
        Make the height of $T^i$ the same height of $T_1$ and $T_2$ by adjusting all external branch length.\;
    }
    Set $L = L \cup \{T^i\}$.\;
    Set $T^{{n \choose 2}+1}=T_1$.\;
}
\Return{the line segments of lines from $T^i$ to $T^{i+1}$ in $L$ for $i = 0, 1, \ldots , {n\choose 2}$.}\;
\caption{Tropical line segment in $\mathcal{U}_n$ written in terms of equidistant trees}\label{alg:trop_line2}
\end{algorithm}

\begin{example}\label{eg:trop_line1}
The input for Algorithm \ref{alg:trop_line2} $T_1$ and $T_2$ are shown in Fig. \ref{fig:trop_line1}.  Their ultrametrics are 
\[
\begin{array}{rcl}
u&=&(0.4, 0.8, 2, 0.8, 2, 2)\\
v&=&(0.8, 0.8, 2, 0.4, 2, 2)\\
\end{array}
\]

Here $\lambda = v - u = (0.4, 0,0,-0.4,0,0)$.  The tropical line segment between $u$ and $v$ is a line segment consisting of the following ultrametrics:
\[
\begin{array}{rcl}
v&=&(0.8, 0.8, 2, 0.4, 2, 2)\\
&& (0.8, 0.8, 2, 0.8, 2, 2)\\
u&=&(0.4, 0.8, 2, 0.8, 2, 2)\\
\end{array}
\]
We sort the elements in $\lambda$ from the smallest to the largest as
\[
\lambda = (-0.4, 0, 0, 0, 0, 0.4).
\]
For $\lambda_{23} = -0.4$, then the tree topology of the tree associated with the ultrametric $(0.8, 0.8, 2, 0.4, 2, 2)$ is $T_2$.  For $\lambda_{i'j'} = 0$ for $(i', j') = (1, 3), (1, 4), (2, 4), (3, 4)$, then the tree topology associated with the ultrametric $(0.8, 0.8, 2, 0.8, 2, 2)$ is the tree with leaves $1, \, 2, \, 3$ are attached to a single interior node.  When $\lambda_{12} = 0.4$, the  tree topology of the tree associated with the ultrametric $(0.4, 0.8, 2, 0.8, 2, 2)$ is $T_1$.  
\begin{figure}[!h]
  \begin{center}
    \includegraphics[width=9cm]{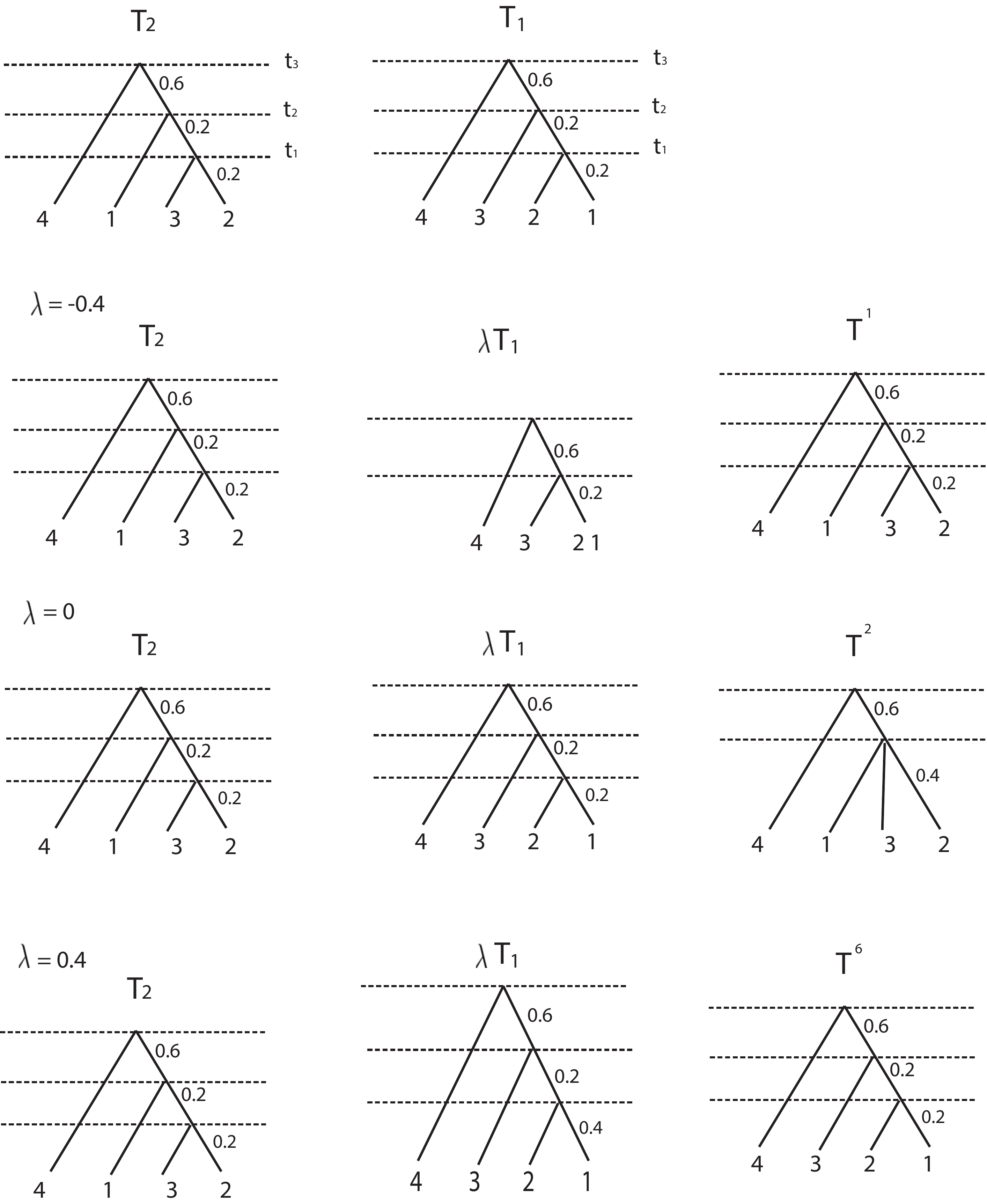}  \qquad
  \end{center}
  \caption{\label{fig:trop_line1}   Trees shown in Example \ref{eg:trop_line1}.}
  \end{figure}
The output from Algorithm \ref{alg:trop_line2} is the set of trees in the left column in Fig.~\ref{fig:trop_line1}.
The tropical line segment between $T_1$ and $T_2$ for this example is a line segment from $T_1$ to $T_2$ via the bending point which is the tree with the ultrametric $(0.8, 0.8, 2, 0.8, 2, 2)$.
\end{example}

\begin{prop}[Proposition 5.2.5 in \cite{MS}]
The  time  complexity  to  compute  the tropical line segment between $T_1$ and $T_2$ with $n$ leaves is $O(n^2\log n) =O(e\log e)$.
\end{prop}

\section{Tree Topologies along a Tropical Line Segment}

Now we consider Problem \ref{prob1}, i.e., tree topologies along with the tropical line segment between two trees.  
In order to solve Problem \ref{prob1}, 
we need to consider relations between an equidistant tree and its ultrametric.  
First we consider when one of the two trees in the end of a tropical line segment is the star tree.  

\begin{definition}
A sequence $\{t_1, \ldots , t_k\}$ is called speciation times in a given equidistant tree $T$ with $n$ leaves if $t_i$ is one half of the $i$th smallest pairwise distance  $d_{ij}$ for any $i, j \in [n]$.
\end{definition}

In terms of an equidistant tree $T$, a speciation time $t_i$ is the height from the leaves to the internal node which is the $i$th smallest branch length from its offspring leaf.

\begin{example}\label{eg:height}
Consider the tree shown in Fig. \ref{fig:height}.  $t_1, \, t_2, \, t_3$ are speciation times. For this tree $t_1 = 0.2$, $t_2 = 0.4$ and $t_3 = 1$.
\begin{figure}[!h]
  \begin{center}
    \includegraphics[width=5.5cm]{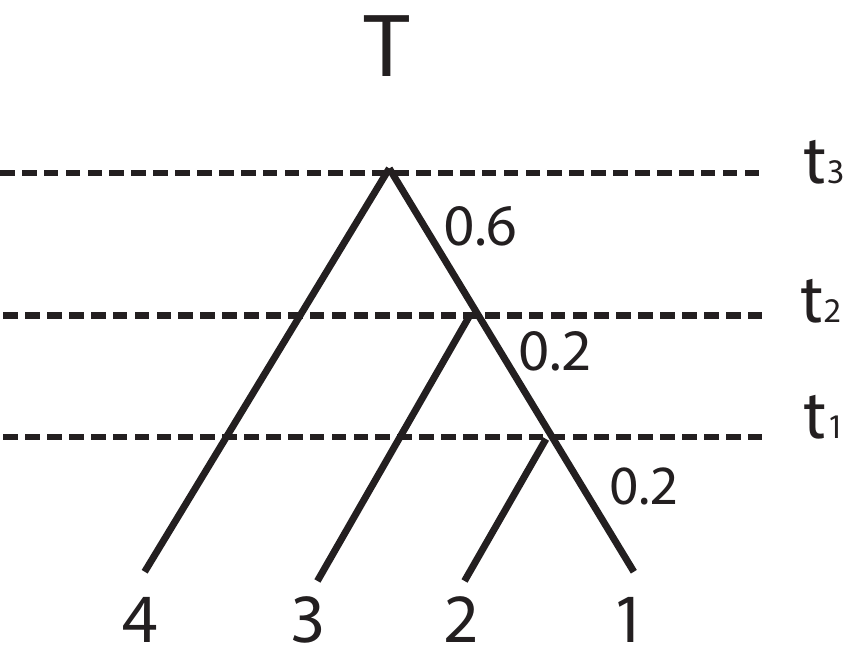}  \qquad
  \end{center}
  \caption{\label{fig:height}   An equidistant tree $T$ for Example \ref{eg:height}.}
  \end{figure}
\end{example}

For an equidistant tree $T$ with $n$ leaves and with a sequence of speciation times $\{t_1, \ldots , t_k\}$, we define a sequence of trees $T^1, \ldots T^k$ associated with the tree $T$, where $k \leq n-1$, such that $T$ and $T^i$ have the same tree topology except that $T^i$ has the sequence of speciation times $\{t_{1}, \ldots , t_k\}$ for $1 \leq i \leq k$.  In general $T^1 = T$.  Also note that $T^k$ is the star tree with $n$ leaves.  

\begin{example}\label{eg:height2}
Consider the tree with $8$ leaves shown in Fig. \ref{fig:height2}.  Then we have $T^2, \, T^3, \, T^4, \, T^5, \, T^6$  as shown in Fig. \ref{fig:height2}. $T^7$ is the star tree.  

The shortest pairwise distance is the distance between leaves 7 and 8.  So $t_1 = 0.2$.  The second shortest pairwise distance is the distance between leaves 3 and 4.  So $t_2 = 0.3$ and by Theorem \ref{trop_line_seg1},  $T^2$ is same as $T$, except the pairwise distance between 7 and 8 which is equal to the pairwise distance between 3 and 4 ($= 2\cdot t_2 = 0.6$).  

The next shortest pairwise distance is the pairwise distance between leaves 5 and 6.  
So, by Theorem \ref{trop_line_seg1}, $T^3$ is same as $T^2$, except the pairwise distances between leaves 7 and 8, between leaves 3 and 4, and between leaves 5 and 6 are equal to $2 \cdot t_3$ where $t_3 = 0.4$.    

Then, the next shortest pairwise distances in $T$ are the pairwise distance between two leaves from $\{2, 3, 4\}$.  So by Theorem \ref{trop_line_seg1}, $T^4$ is same as $T^3$, except leaves 2, 3, and 4 form a polytomy with its height $t_4 = 0.6$, and the pairwise distances between leaves 5 and 6 and between 7 and 8 are $2 \cdot t_4$. 

The next shortest pairwise distances are the pairwise distance between two leaves from $\{5, 6, 7, 8\}$.  So by Theorem \ref{trop_line_seg1}, $T^5$ is same as $T^4$, except leaves 5, 6, 7, and 8 form a polytomy with its height $t_5 = 0.7$ and leaves 2, 3, and 4 form a polytomy with  its height $t_5 = 0.7$.  The next shortest pairwise distances are the pairwise distances between leaves $i$ and $j$ for $i, j \in \{1, 2, 3, 4\}$.  Thus, by Theorem \ref{trop_line_seg1}, $T^6$ is the equidistant tree where the leaves 1, 2, 3, 4 form a polytomy with its height 1.0 and the leaves 5, 6, 7, 8 form a polytomy with its height 1.0.

The tropical line segment from the star tree to $T$ is the line segment of ultrametrics computed from the trees $T, \, T^2, \, T^3, \, T^4, \, T^5, \, T^6$ as shown in Fig.~\ref{fig:height2} and  $T^7$ which is the star tree.

\begin{figure}[!h]
  \begin{center}
    \includegraphics[width=9cm]{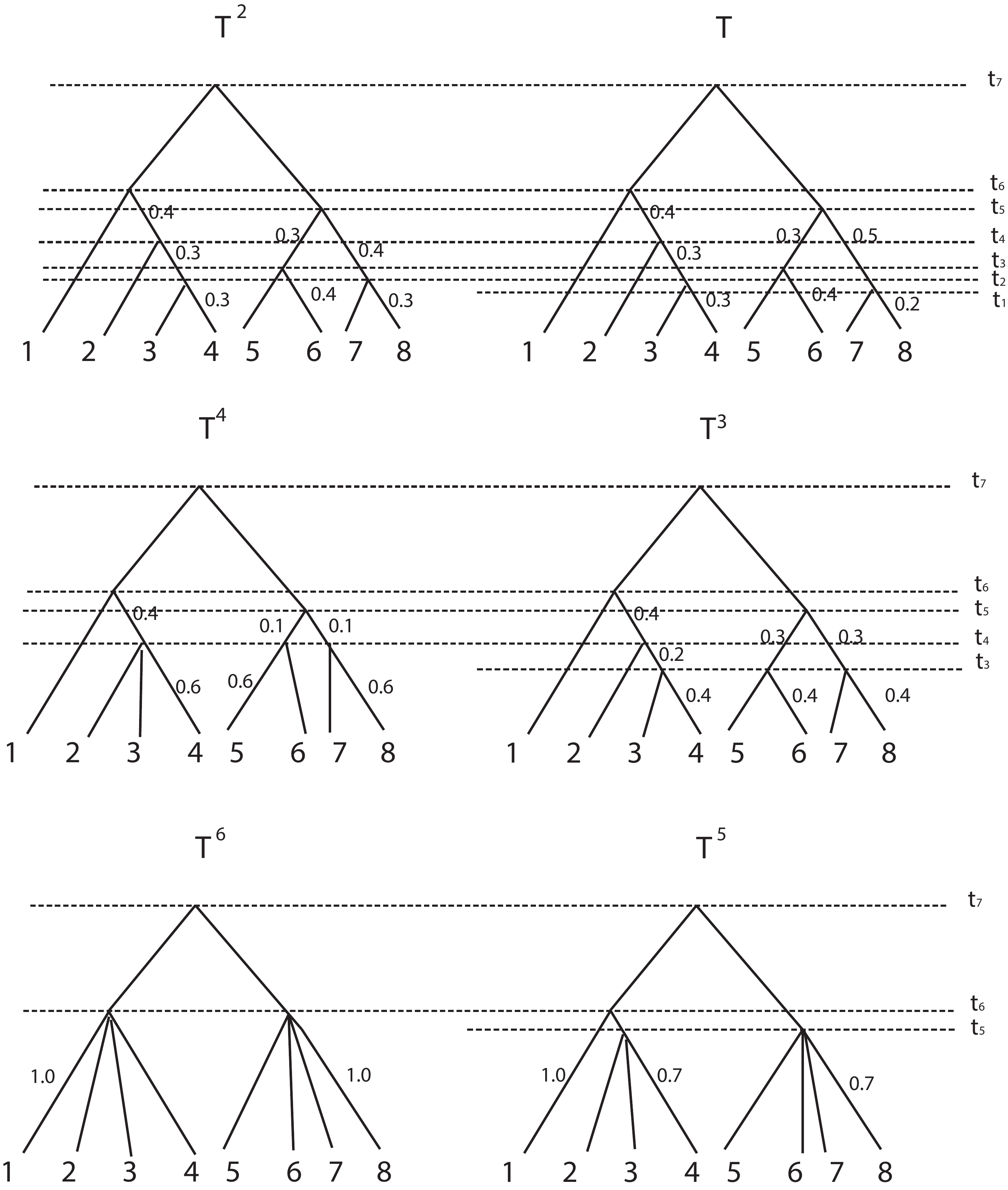}  \qquad
  \end{center}
  \caption{\label{fig:height2}   Trees shown in Example \ref{eg:height2}.}
  \end{figure}
\end{example}

\begin{thm}\label{trop_line_seg1}
Suppose we have an equidistant tree $T$ with $n$ leaves and a sequence of speciation times $\{t_1, \ldots , t_k\}$ where $k \leq n-1$.  Then the tropical line segment from $T$ in $\mathcal{U}_n$ to the origin, i.e., the star tree with $n$ leaves, is the line segments of lines between the ultrametrics of trees $T^{i}, T^{i+1}$, for $i = 1, \ldots , k-1$ associated with the tree $T$.  
\end{thm}
\begin{proof}
Follow Algorithms \ref{alg:trop_line1} and \ref{alg:trop_line2}.  In this case we have $u_{ij} = 0$ for all $i, j \in \{1, 2, \ldots , n\}$.  Thus $\lambda = v$.  
At the {\bf for-loop} in Algorithm  \ref{alg:trop_line1} for the iteration $i$, i.e., the $i$th smallest pairwise distance with the pair $(i_1, i_2)$ in $\lambda = v$
\begin{equation}\label{ultrametric1}
\max\{\lambda_{i_1 i_2}, v_{kl}\} = \max\{v_{i_1i_2}, v_{kl}\} = \begin{cases}
v_{i_1i_2} & \mbox{if } v_{i_1i_2} \geq v_{kl}\\
v_{kl} & \mbox{otherwise,}
\end{cases}
\end{equation}
for all $k, l \in \{1, \ldots , n\}$.  Note that the equidistant tree  corresponds to this ultrametric in \eqref{ultrametric1} is $T^i$.  
\end{proof}

\begin{example}\label{eg:height3}
Consider the tree with $8$ leaves shown in Fig.~\ref{fig:height2}.  We are interested in drawing a tropical line segment from the tree $T$ to the origin, i.e., the star tree with $8$ leaves.  

First we compute $\lambda = u - 0 = u$, where $u$ is the ultrametric of $T$.  Then we order $\lambda$ from the smallest to the largest, i.e.,
\[
\lambda = (2\cdot t_1, \ldots , 2\cdot t_7).
\]
Now, we iterate for each element in $\lambda$ from the smallest to the largest. 

For $2\cdot t_1$, we have 
\[
u_{ij} = \max\{2\cdot t_1, u_{ij}\}
\]
for $i, j \in X$ with $i < j$.  Thus the equidistant tree for $2\cdot t_1$ is $T$.  

For $2\cdot t_2$, we have
\[
u_{ij} = \max\{2\cdot t_2, u_{ij}\}
\]
for $(i, j) \in (X \times X) - \{(7, 8)\}$ and
\[
2\cdot t_2 = \max\{2\cdot t_2, u_{78}\}.
\]
Then, this ultrametric gives the equidistant tree $T^2$ in Fig.~\ref{fig:height2}.

For $2\cdot t_3$, we have
\[
u_{ij} = \max\{2\cdot t_3, u_{ij}\}
\]
for $(i, j) \in (X \times X) - \{(3, 4), (7, 8)\}$ and
\[
\begin{array}{c}
2\cdot t_3 = \max\{2\cdot t_3, u_{34}\},\\
2\cdot t_3 = \max\{2\cdot t_3, u_{78}\}.
\end{array}
\]
Then, this ultrametric gives the equidistant tree $T^3$ in Fig.~\ref{fig:height2}.

For $2\cdot t_4$, we have
\[
u_{ij} = \max\{2\cdot t_4, u_{ij}\}
\]
for $(i, j) \in (X \times X) - \{(3, 4), (5, 6), (7, 8)\}$ and
\[
\begin{array}{c}
2\cdot t_4 = \max\{2\cdot t_4, u_{34}\},\\
2\cdot t_4 = \max\{2\cdot t_4, u_{56}\},\\
2\cdot t_4 = \max\{2\cdot t_4, u_{78}\}.
\end{array}
\]
Then, this ultrametric gives the equidistant tree $T^4$ in Fig.~\ref{fig:height2}.

For $2\cdot t_5$, we have
\[
u_{ij} = \max\{2\cdot t_5, u_{ij}\}
\]
for $(i, j) \in (X \times X) - \{(2, 3), (2, 4), (3, 4), (5, 6), (5, 7), (5, 8), (6, 7), (6, 8), (7, 8)\}$ and
\[
\begin{array}{c}
2\cdot t_5 = \max\{2\cdot t_5, u_{23}\},\\
2\cdot t_5 = \max\{2\cdot t_5, u_{24}\},\\
2\cdot t_5 = \max\{2\cdot t_5, u_{34}\},\\
2\cdot t_5 = \max\{2\cdot t_5, u_{56}\},\\
2\cdot t_5 = \max\{2\cdot t_5, u_{57}\},\\
2\cdot t_5 = \max\{2\cdot t_5, u_{58}\},\\
2\cdot t_5 = \max\{2\cdot t_5, u_{67}\},\\
2\cdot t_5 = \max\{2\cdot t_5, u_{68}\},\\
2\cdot t_5 = \max\{2\cdot t_5, u_{78}\}.
\end{array}
\]
Then, this ultrametric gives the equidistant tree $T^5$ in Fig.~\ref{fig:height2}.

For $2\cdot t_6$, we have
\[
u_{ij} = \max\{2\cdot t_5, u_{ij}\}
\]
for $(i, j) \in (\{1, 2, 3, 4\} \times \{5, 6, 7, 8\})$ and
\[
\begin{array}{c}
2\cdot t_6 = \max\{2\cdot t_6, u_{ij}\},\\
2\cdot t_6 = \max\{2\cdot t_6, u_{kl}\},\\
\end{array}
\]
where $(i, j) \in \{1, 2, 3, 4\} \times \{1, 2, 3, 4\}$ and $(k, l) \in \{5, 6, 7, 8\} \times \{5, 6, 7, 8\}$.
Then, this ultrametric gives the equidistant tree $T^6$ in Fig.~\ref{fig:height2}.

Finally, for $2\cdot t_7$, we have
\[
\begin{array}{c}
2\cdot t_7 = \max\{2\cdot t_7, u_{ij}\},\\
\end{array}
\]
for all $(i, j) \in X \times X$.  Thus we have the result shown in Fig.~\ref{fig:height2}.

\end{example}

With this theorem, we can solve our problem for $n = 3$.
\begin{lemma}\label{lem3leaves}
Suppose  $T_1, T_2$ are equidistant trees with $n=3$ leaves such that $T_1, T_2$ have different tree topologies.  Then tree topologies along $\Gamma_{T_1, T_2}$ change from the tree topology of $T_1$ to the star tree, and then change from the star tree to the tree topology of $T_2$.
\end{lemma}

\begin{proof}
Without loss of generality, $T_1$ and $T_2$ have tree topologies shown in Fig \ref{fig:3leaves}. Let $u^1 = (u^1_{12}, u^1_{13}, u^1_{23})$ be an ultrametric for $T_1$ and $u^2 = (u^2_{12}, u^2_{13}, u^2_{23})$ be an ultrametric for $T_2$.  Then we have
\[
\begin{array}{c}
     u^1_{23} \leq u^1_{12} = u^1_{13}\\
     u^2_{13} \leq u^2_{12} = u^2_{23}.\\
\end{array}
\]
\begin{figure}
    \centering
    \includegraphics[width=0.6\textwidth]{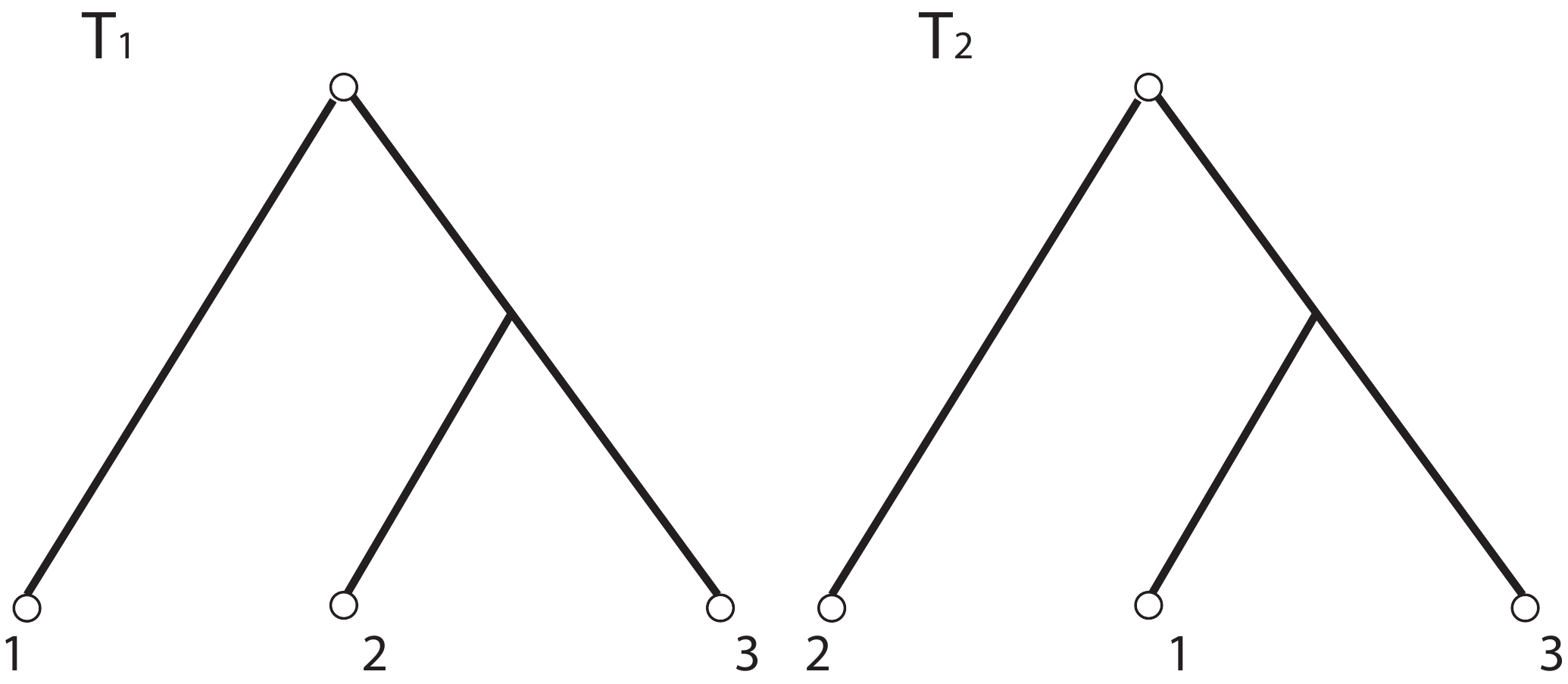}
    \caption{$T_1$ and $T_2$ in the proof of Lemma \ref{lem3leaves}} 
    \label{fig:3leaves}
\end{figure}

Then we have $u^2 - u^1 = (u^2_{12} - u^1_{12}, u^2_{13} - u^1_{13}, u^2_{23} - u^1_{23})$.  Then since all elements in an ultrametric are non-negative, we have
\[
u^2_{13} - u^1_{13} \leq u^2_{12} - u^1_{12} \leq   u^2_{23} - u^1_{23}. 
\]

Using Algorithm \ref{alg:trop_line2},
for $\lambda_{13} = u^2_{13} - u^1_{13}$, then we have $T_2$.  Also for $\lambda_{23} = u^2_{23} - u^1_{23}$, then we have $T_1$.  So we have to see the tree topology if $\lambda_{12} = u^2_{12} - u^1_{12}$. Then we have
\[
\left(
\begin{array}{c}
     \max\{u^1_{12} + u^2_{12} - u^1_{12}, u^2_{12}\}\\
     \max\{u^1_{13} + u^2_{12} - u^1_{12}, u^2_{13}\}\\
     \max\{u^1_{23} + u^2_{12} - u^1_{12}, u^2_{23}\}\\
\end{array}
\right)
= 
\left(
\begin{array}{c}
     u^2_{12}\\
     u^2_{12}\\
     u^2_{23}\\
\end{array}
\right)
=
\left(
\begin{array}{c}
     u^2_{12}\\
     u^2_{12}\\
     u^2_{12}\\
\end{array}
\right)
\]
which is the star tree.  With Theorem \ref{trop_line_seg1} we are done.  
\end{proof}

\begin{remark}
Suppose $T_1, \, T_2 \in \mathcal{U}_3$ with $n = 3$. Then,
the tree topologies along the geodesic between equidistant trees $T_1$ and $T_2$ under the BHV metric are the same as tree topologies along $\Gamma_{T_1, T_2}$.
\end{remark}


\begin{example}\label{ex:topologies}
In this example, we map a tropical line segment between two points in $\mathcal{U}_4$ onto the BHV treespace for rooted trees with $n = 4$ leaves shown in Fig.~\ref{fig:topologies}.  $T_1, T_2$ have the tree topologies $((1, 2), (3,4))$ and $(((1, 2), 3), 4)$  written in the Newick format \cite{newick}.  Note that with this map, we use the tree space coordinate of the BHV metric, but technically this is not the tree space defined by the BHV metric. 
\begin{figure}[h!]
\centering     
\subfigure{\label{fig:a1}\includegraphics[width=0.8\textwidth]{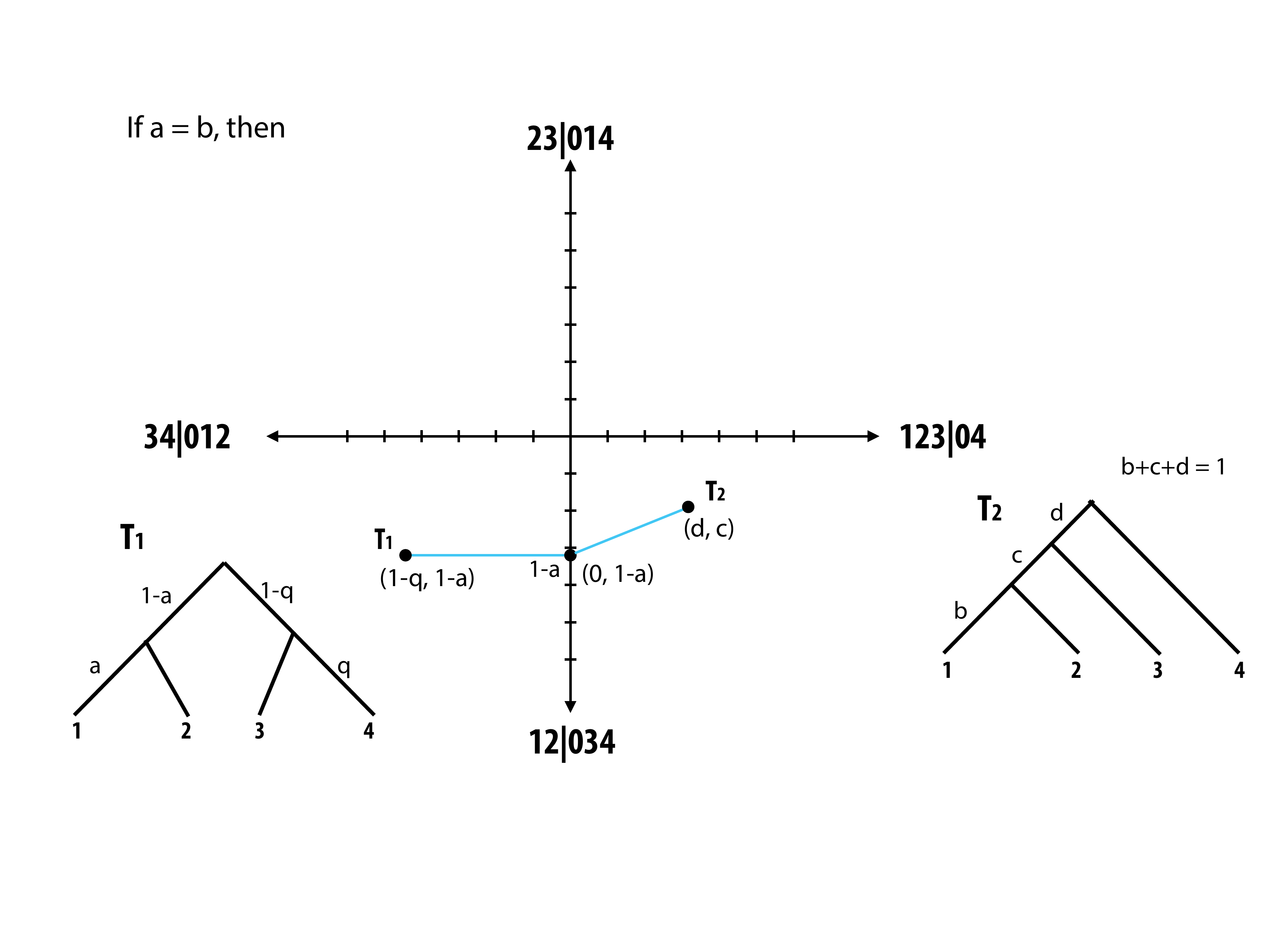}}
\vskip -1in
\subfigure{\label{fig:b1}\includegraphics[width=0.8\textwidth]{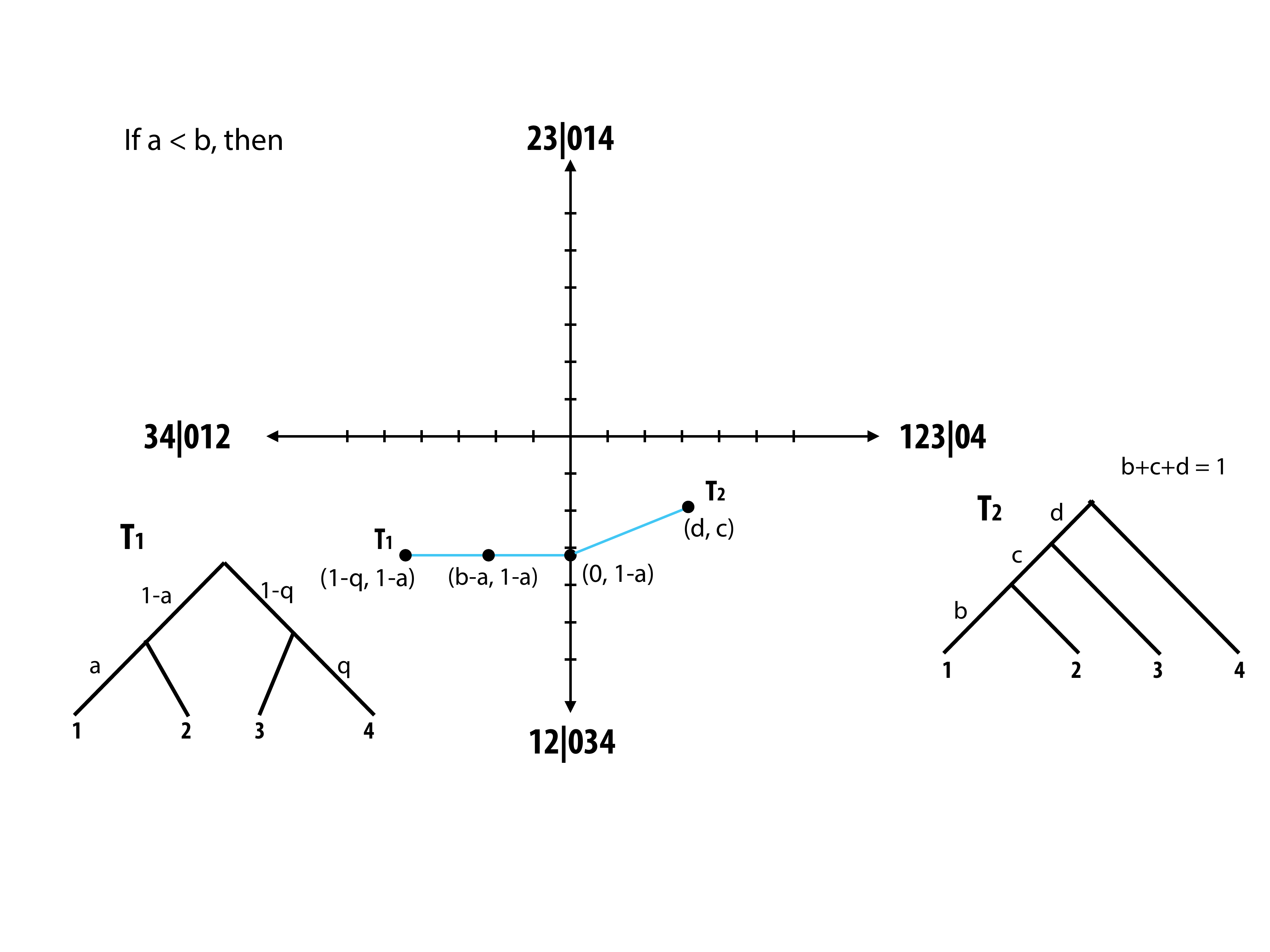}}
\vskip -1in
\subfigure{\label{fig:c1}\includegraphics[width=0.8\textwidth]{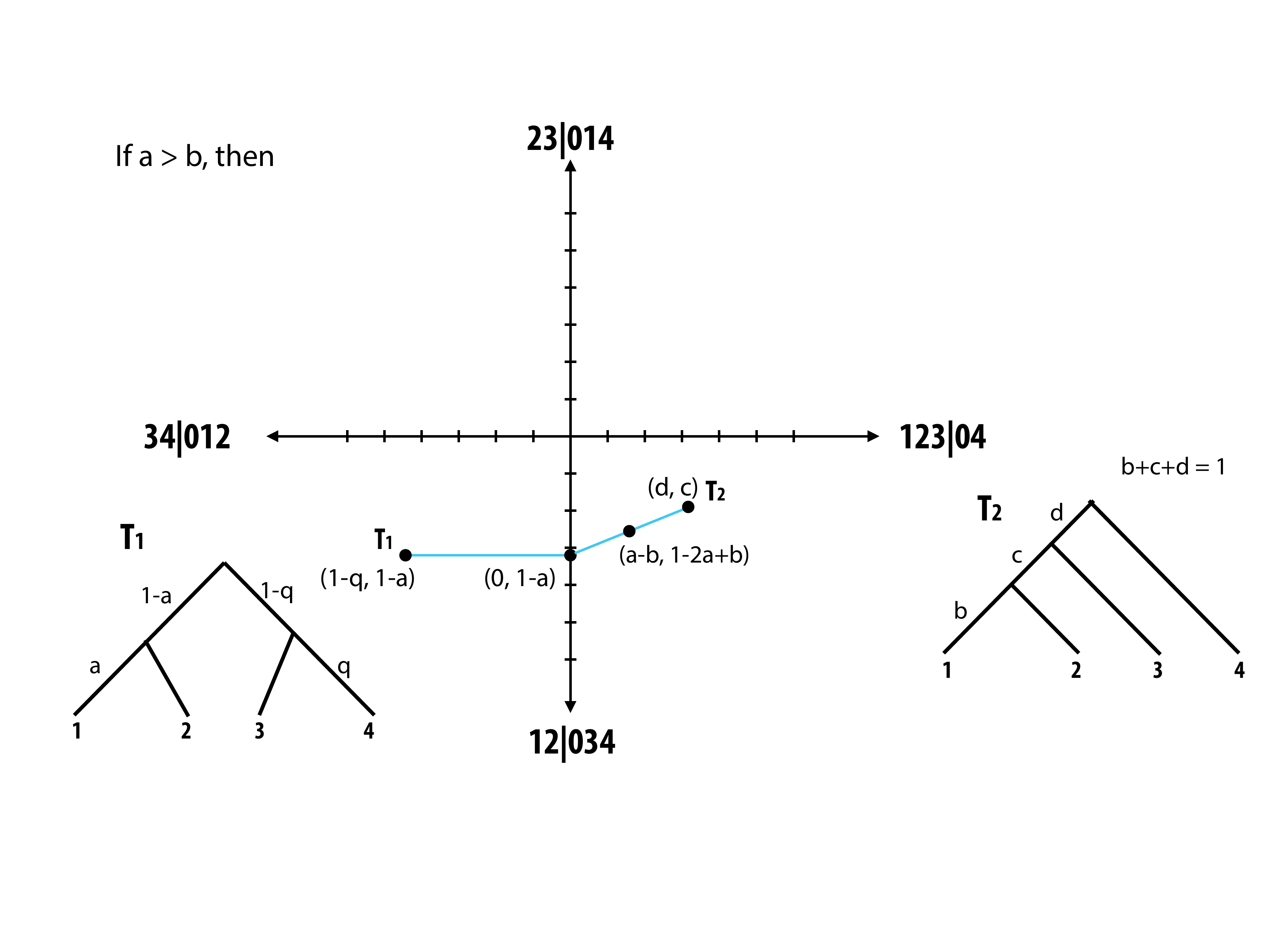}}
\vskip -0.5in
\caption{The tree topologies for the ends of the tropical line segment are $((1, 2), (3,4))$ and $(((1, 2), 3), 4)$ in Example \ref{ex:topologies}.}\label{fig:topologies}
\end{figure}
\end{example}



\begin{definition}
Suppose we have an equidistant phylogenetic tree $T$ with the leave set $X$ and the ultrametric $u = (u_{12}, \ldots , u_{n-1n})$.
A {\em subtree} of $T$ with leaves $X_0 = \{i_1, \ldots , i_{n_0}\} \subset X$, where $n_0 \leq n$, is an equidistant tree constructed from an ultrametric $u^0$ such that
\[
u^0 = (u_{i_1i_2}, \ldots , u_{i_{n_0 - 1}i_{n_0}}).
\]
\end{definition}

\begin{definition}\label{def:clade}
Suppose we have an equidistant phylogenetic tree $T$ with the leave set $X$.  A {\em clade} of $T$ with leaves $X_0 \subset X$ is an equidistant tree constructed from $T$ by adding all common ancestral interior nodes of any combinations of only leaves $X_0$ and excluding common ancestors including any leaf from $X - X_0$   in $T$, and all edges in $T$ connecting to these ancestral interior nodes and leaves  $X_0$.  
\end{definition}

\begin{remark}
Suppose we have an equidistant phylogenetic tree $T$ with the leave set $X$. 
A clade of an equidistant tree $T$ with leave set $X_0 \subset X$ is a subtree of $T$ with the leaves $X_0$.
\end{remark}

\begin{example}\label{ex:subtree1}
Suppose we have two equidistant trees $T_1$ and $T_2$ with the leaf set $\{S_1, S_2, S_3, S_4, S_5\}$ shown in Fig.~\ref{fig:subtree_example}.  Let $X_0 = \{S_1, S_2, S_3\}$.  Suppose we have interior nodes $\{y_1, y_2, y_3, y_4\}$ for $T_1$ such that the root of $T_1$ is $y_4$ and we have interior nodes $\{y_1, y_2, y_3, y_5\}$ for $T_2$ such that the root of $T_2$ is $y_3$. 

$T'_1$ is a clade of $T_1$ with leaves $X_0$ since $y_1$ is a common ancestor of $S_1$ and $S_2$, and $y_2$ is a common ancestor of $S_1$, $S_2$ and $S_3$, but we exclude an interior node $y_3$ from $T_1$ since $y_3$ is a common ancestor of $\{S_1, S_2, S_3, S_4\}$, where $S_4 \not \in X_0$.  

Similarly, $T'_2$ is a clade of $T_2$ with leaves $X_0$ since $y_1$ is a common ancestor of $S_1$ and $S_2$ and $y_2$ is a common ancestor of $S_1$, $S_2$ and $S_3$, but we exclude an interior node $y_3$ from $T_2$ since $y_3$ is a common ancestor of $\{S_1, S_2, S_3, S_4, S_5\}$, where $S_4, S_5 \not \in X_0$.  

Fig.~\ref{fig:subtree_example2} shows ultrametrics $u$ and $v$ associate with equidistant trees $T_1$ and $T_2$, respectively.  Then $u^0$ and $v^0$ are ultrametrics of clades of $T_1$ and $T_2$, respectively, with leaves $X_0$.

\begin{figure}
    \centering
    \includegraphics[width=\textwidth]{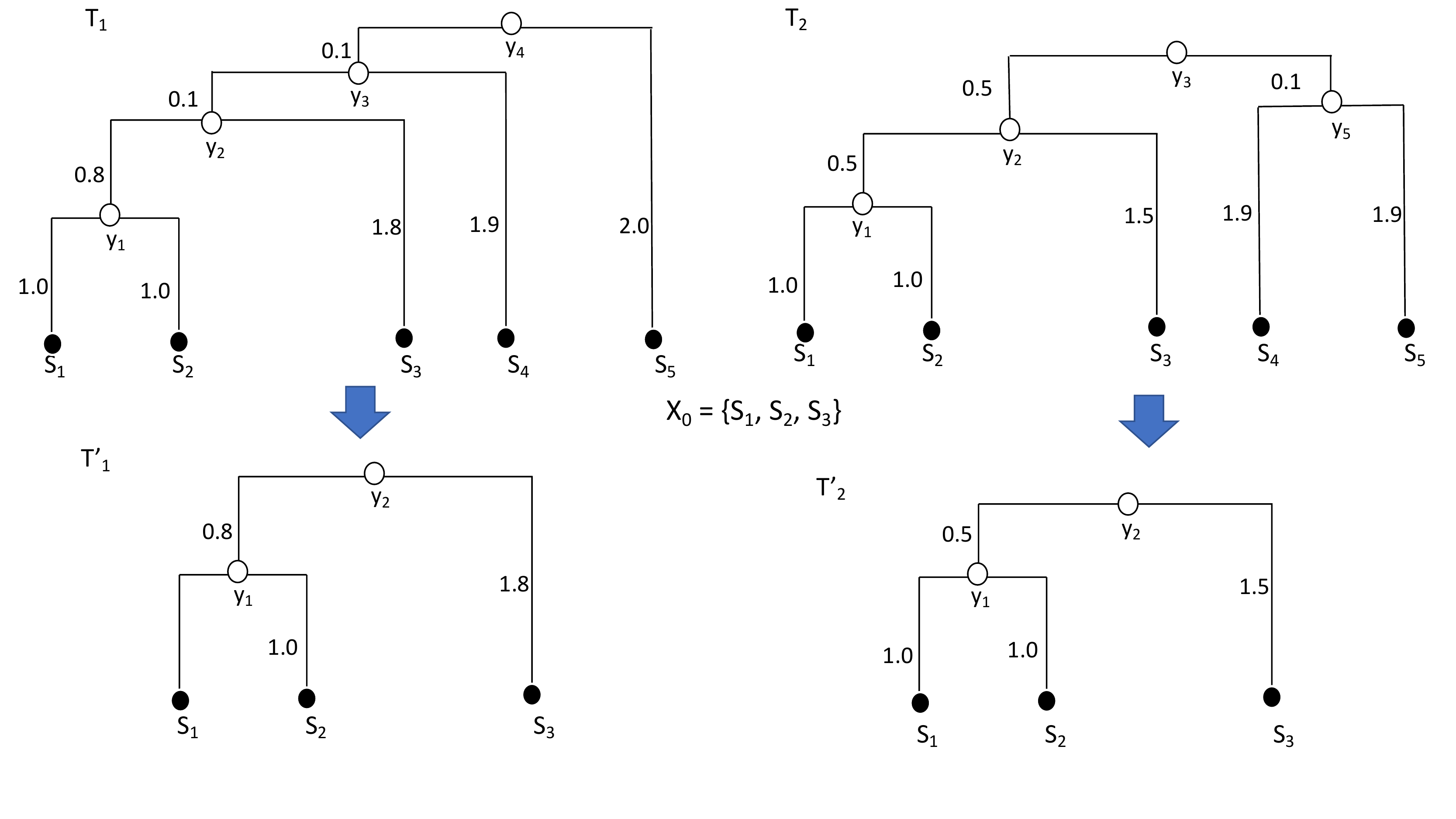}
    \caption{Example of clades in equidistant trees written in Example \ref{ex:subtree1}.}
    \label{fig:subtree_example}
\end{figure}

\begin{figure}
    \centering
    \includegraphics[width=\textwidth]{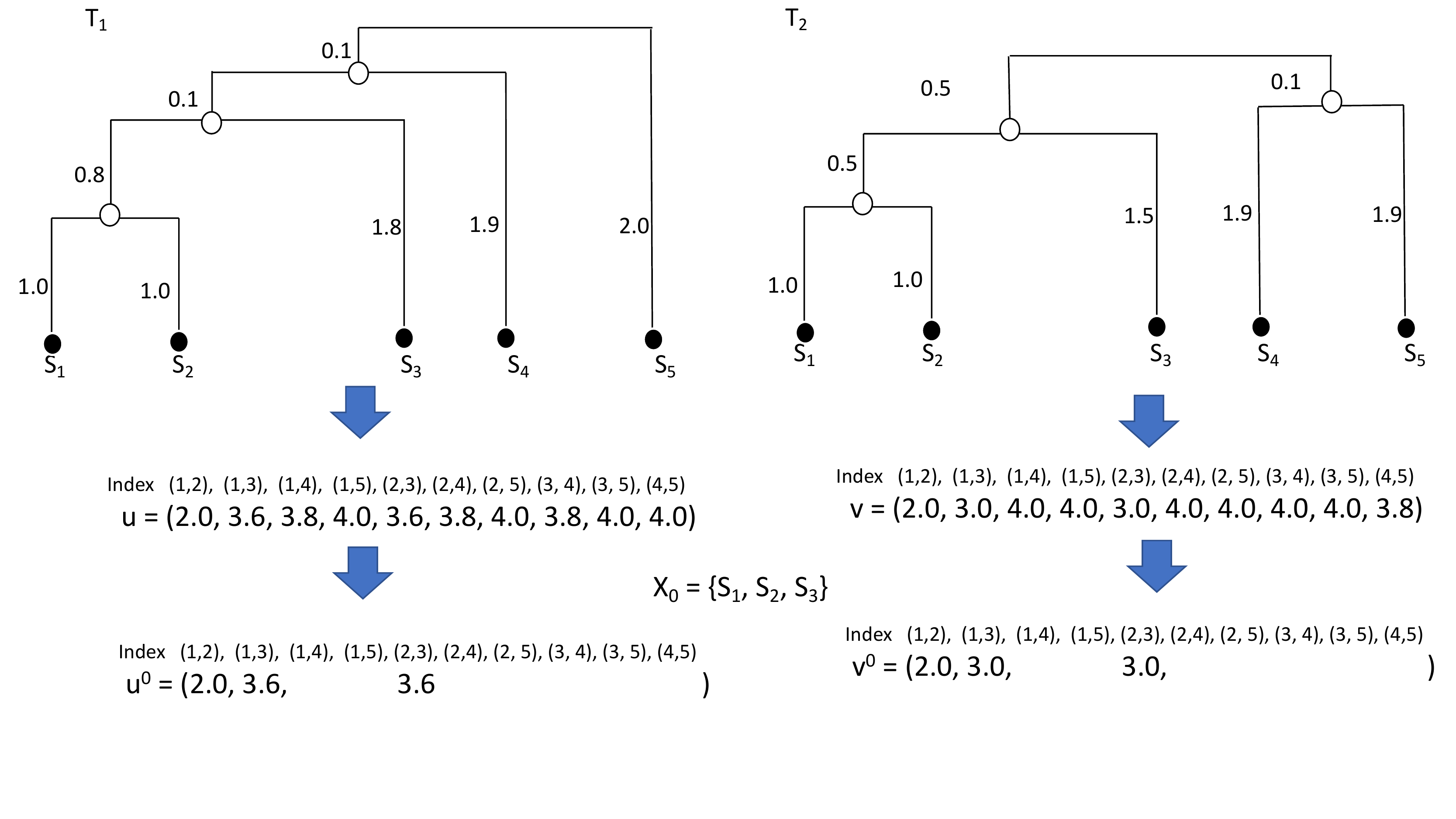}
    \caption{Example of ultrametrics of clades in equidistant trees written in Example \ref{ex:subtree1}. }
    \label{fig:subtree_example2}
\end{figure}
\end{example}

Now we consider the tropical line segment with two trees which share that same tree topology of their clades with leaves $X_0 \subset X$.  In order to see how tree topologies change over the tropical line segment when these two trees share that same tree topology of their clades with leaves $X_0 \subset X$, we 
let 
\[
(\argmax)^i\{x\}
\]
be the set of indices of entries of $x$ achieving the $i$th largest value among entries of a vector $x$.
Then, we have the following lemma.
\begin{lemma}\label{lemma:topology1}
Suppose we have equidistant trees $T_1$ and $T_2$ with their ultrametrics $u^1$ and $u^2$, respectively.  
If
\[
(\argmax)^i\{u^1\} = (\argmax)^i\{u^2\}
\]
for all $0 < i < e$, then $T_1$ and $T_2$ have the same tree topology.
\end{lemma}
\begin{proof}
Suppose we have
\[
(\argmax)^i\{u^1\} = (\argmax)^i\{u^2\}
\]
for all $0 < i < e$.  Then we have
\[
\argmax\{u^1_{ij}, u^1_{ik}, u^1_{jk}\} = \argmax\{u^2_{ij}, u^2_{ik}, u^2_{jk}\}
\]
where $i, j, k \in X$ are distinct labels of leaves in $T_1$ and $T_2$.  Therefore, $T_1$ and $T_2$ have the same tree topologies by Theorem 1 in \cite{Yoshida}.

\end{proof}

\begin{thm}\label{same_tree_topology2}
Suppose $T_1$, $T_2$ are equidistant trees on $n = |X|$ leaves and $X_0$ is a subset of leaves $X$ which forms a clade in both $T_1$ and $T_2$, with the same tree topology in $T_1$, $T_2$. Then for any tree $T$ on the tropical line segment from $T_1$ to $T_2$, $X_0$ is also a clade of $T$ with the same tree topology as in $T_1$, $T_2$.
\end{thm}
\begin{proof}
Let $u = (u_{11}, u_{12}, \ldots , u_{n-1n})$ and $v = (v_{11}, v_{12}, \ldots , v_{n-1n})$ be ultrametrics associated to $T_1$ and $T_2$, respectively. 
Let $w = u \odot \alpha \oplus v$ be an ultrametric for any tree $T$ on the tropical line segment from $u$ to $v$.

\begin{enumerate}
        \item The restriction of $T_1$ to the leaves $X_0$ is again an equidistant tree.
        
        \begin{proof}
            This is shown in Definition \ref{def:clade}.
        \end{proof}
        
        \item $X_0$ is a clade of $T$.
        
        \begin{proof}
            A clade is the collection of all descendant leaves of an internal vertex. $X_0$ is a clade of $T$ if for any $i,j \in X_0$ and $k \notin X_0$, we have
            $$w_{ij} < w_{ik} = w_{jk}.$$
            By definition of $w$,
            $$w_{ij} = \max(u_{ij} + \alpha, v_{ij}),$$
            $$w_{ik} = \max(u_{ik} + \alpha, v_{ik}).$$
            Since $X_0$ forms a clade in $T_1$ and $T_2$ by assumption, $v_{ij} < v_{ik}$ and $u_{ij} + \alpha < u_{ik} + \alpha$, so $w_{ij} < w_{ik}$. This proves that $X_0$ forms a clade in every intermediate tree $w$.
        \end{proof}
        
        \item $T\vert_{X_0}$ has the same tree topology as $T_1\vert_{X_0}$ and $T_2\vert_{X_0}$.
        
        \begin{proof}
            Let $i,j,k \in X_0$, and without loss of generality, say
            $$\max(u_{ij}, u_{ik}, u_{jk}) = u_{ik} = u_{jk},$$
            $$\max(v_{ij}, v_{ik}, v_{jk}) = v_{ik} = v_{jk}$$
            (since we assume that $u$, $v$ have the same tree topology when restricted to $X_0$, we know that the indices achieving the max are the same). By definition of $w$, we have
            $$w_{ij} = \max(u_{ij} + \alpha, v_{ij}),$$
            $$w_{ik} = \max(u_{ik} + \alpha, v_{ik}),$$
            $$w_{jk} = \max(u_{jk} + \alpha, v_{jk}).$$
            We want to show that
            $$w_{ij} \leq w_{ik} = w_{jk}$$
            with equality if and only if $u_{ij} = u_{ik} = u_{jk} = v_{ij} = v_{ik} = v_{jk}$ Define:
            $$\lambda_1 = v_{ij} - u_{ij},$$
            $$\lambda_2 = v_{ik} - u_{ik} = v_{jk} - u_{jk}.$$
            There are four cases to consider:
            \begin{enumerate}
                \item \textit{Case 1:} $\lambda_1, \lambda_2 > \alpha$. In this case, $w = v$, and we are done.
                \item \textit{Case 2:} $\lambda_1 \geq \alpha \geq \lambda_2$. In this case,
                \begin{align*}
                    w_{ij} &= \max(u_{ij} + \alpha, v_{ij}) = u_{ij} + \alpha ,\\
                    w_{ik} &= \max(u_{ik} + \alpha, v_{ik}) = v_{ik} ,\\
                    w_{jk} &= \max(u_{jk} + \alpha, v_{jk}) = v_{jk} .
                \end{align*}
                Then
                $$w_{ij} = u_{ij} + \alpha \leq u_{ik} + \alpha \leq v_{ik} = v_{jk} = w_{ik} = w_{jk}$$
                and the leftmost inequality is strict if $u_{ij} < u_{ik} = u_{jk}$, so we are done.
                \item \textit{Case 3:} $\lambda_2 \geq \alpha \geq \lambda_1$. In this case,
                \begin{align*}
                    w_{ij} &= \max(u_{ij} + \alpha, v_{ij}) = v_{ij} ,\\
                    w_{ik} &= \max(u_{ik} + \alpha, v_{ik}) = u_{ik} + \alpha ,\\
                    w_{jk} &= \max(u_{jk} + \alpha, v_{jk}) = u_{jk} + \alpha .
                \end{align*}
                Then
                $$w_{ij} = v_{ij} \leq u_{ij} + \alpha < u_{ik} + \alpha = u_{jk} + \alpha = w_{ik} = w_{jk}$$
                and the rightmost inequality is strict if $u_{ij} < u_{ik} = u_{jk}$, so we are done.
                \item \textit{Case 4:} $\alpha > \lambda_1, \lambda_2$. In this case $w = u$, and we are done.
            \end{enumerate}
        \end{proof}
    \end{enumerate}
This completes the proof.
\end{proof}

\begin{example}\label{ex:subtree2}
Suppose we have equidistant trees $T_1$ and $T_2$ from Example \ref{ex:subtree1}.  Then, the tropical line segment $\Gamma_{T_1, T_2}$ is shown in Fig.~\ref{fig:subtree_example3}.  The tree topology of the clades of trees with leaves $X_0$ on the tropical line segment $\Gamma_{T_1, T_2}$ are the same as these of clades in $T_1$ and $T_2$ with leaves $X_0$.

\begin{figure}
    \centering
    \includegraphics[width=\textwidth]{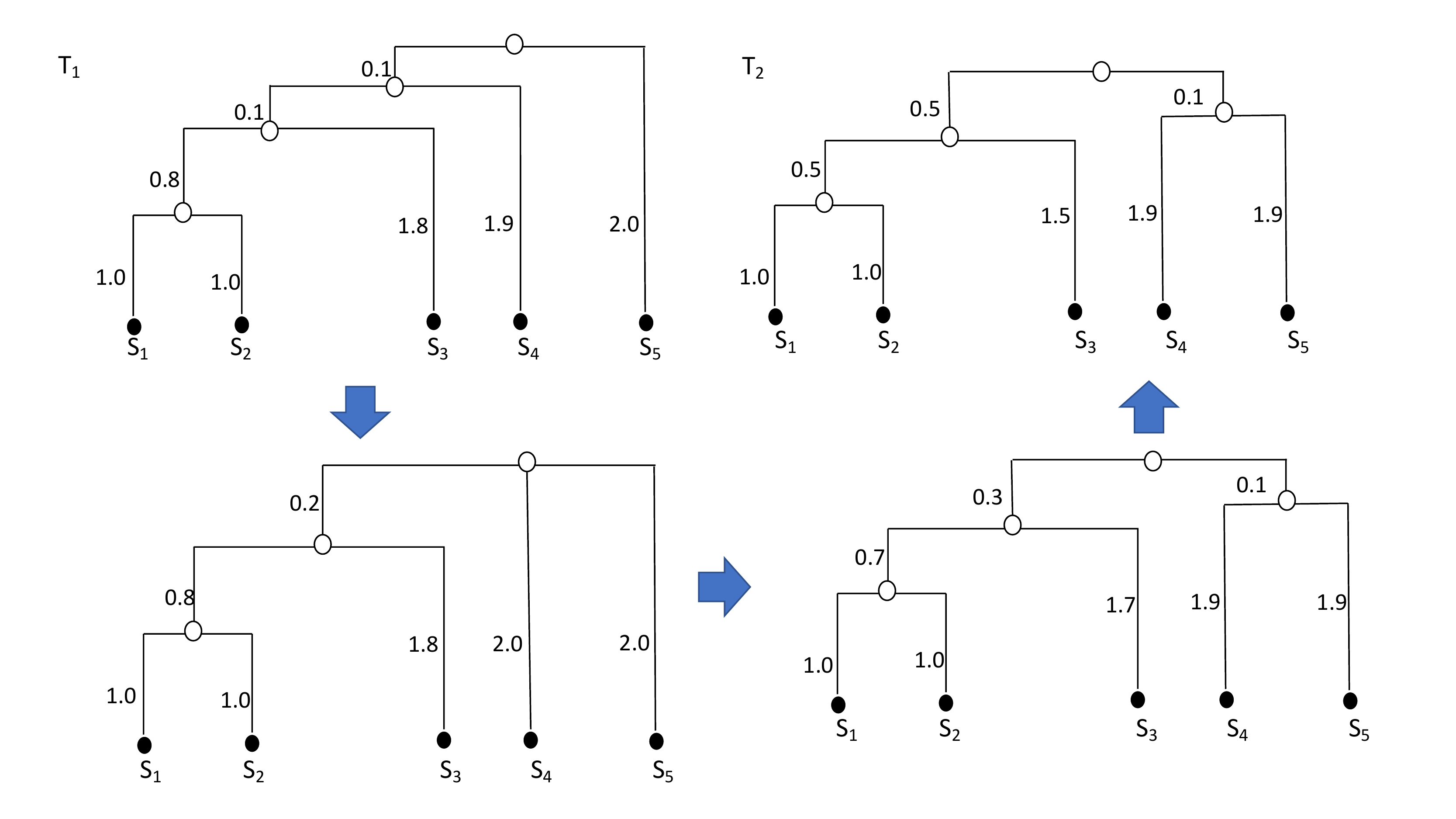}
    \caption{The tropical line segment $\Gamma_{T_1, T_2}$ from $T_1$ and $T_2$ shown in Fig.~\ref{fig:subtree_example}.}
    \label{fig:subtree_example3}
\end{figure}
\end{example}

\begin{definition}
For a rooted phylogenetic tree, a {\em nearest neighbor interchange (NNI)} is an operation of a  phylogenetic tree to change its tree topology by picking three mutually exclusive leaf sets $X_1, X_2, X_3 \subset X$ and changing a tree topology of the clade, possibly the whole tree, consisting with three distinct clades with leaf sets $X_1$, $X_2$, and $X_3$ shown in Fig.~\ref{fig:NNI}.  One NNI move is one of these tree moves shown in Fig.~\ref{fig:NNI}.
\begin{figure}[h!]
    \centering
    \includegraphics[width=\textwidth]{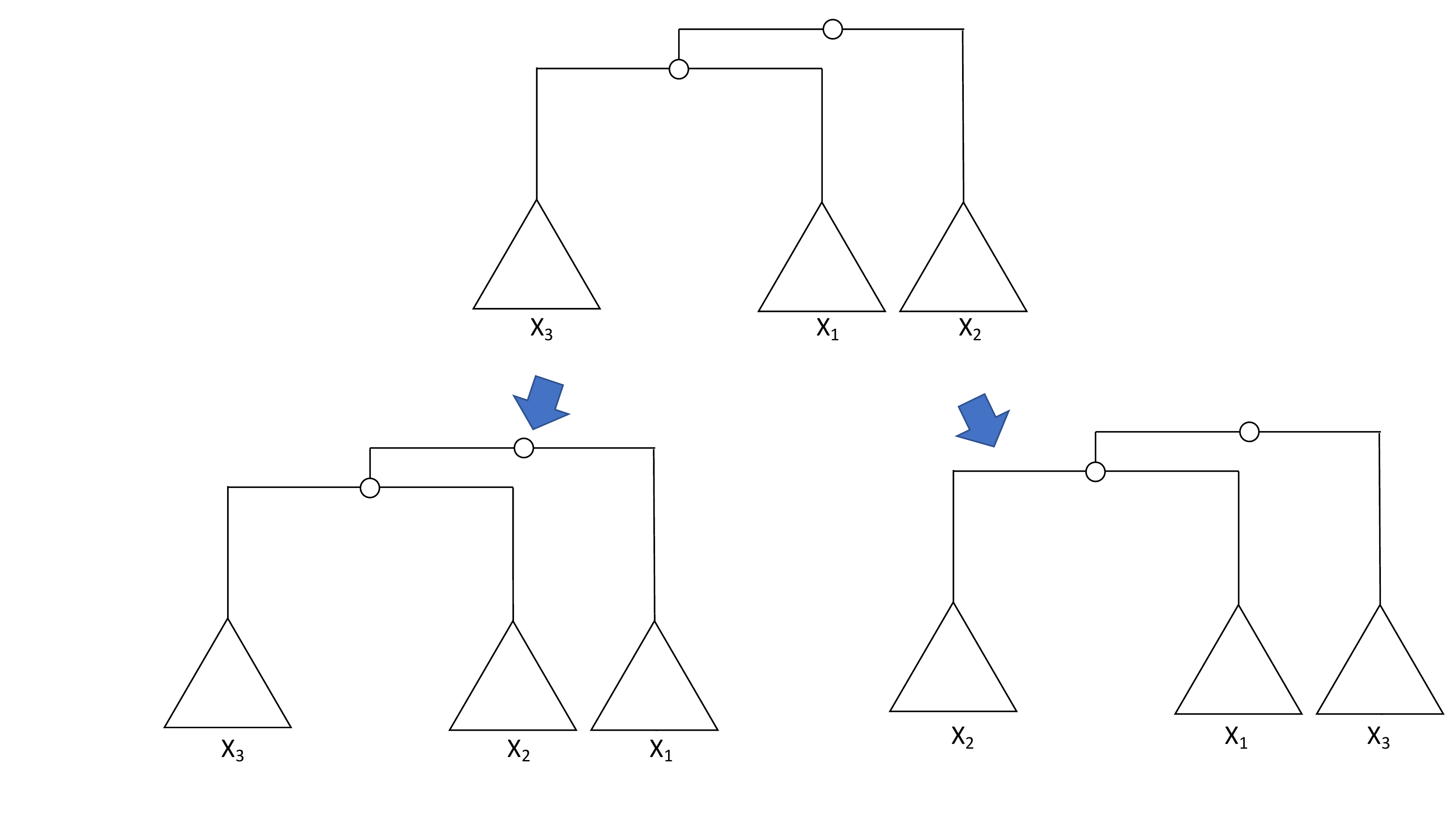}
    \caption{A NNI move is one of these two possible moves.  Triangles with label $\{X_1, X_2, X_3\}$ are clades with leaves $X_1$, $X_2$, and $X_3$, where $X_1, X_2, X_3 \subset X$ are mutually exclusive subsets from $X$.}
    \label{fig:NNI}
\end{figure}
\end{definition}


\begin{thm}\label{th:nni1}
Let $T_1, \, T_2$ be equidistant trees with leaves $X$ such that the tree topology of $T_1$ and the tree topology of $T_2$ are different by only one NNI move.  Then tree topologies on $\Gamma_{T_1, T_2}$ have the same tree topology of $T_1$ or $T_2$ with possible $0$ branch lengths. 
\end{thm}
\begin{proof}
Suppose $T_1, \, T_2$ are equidistant trees with leaves $X$ such that the tree topology of $T_1$ and the tree topology of $T_2$ differs by only one NNI move.  By the definition of an NNI move, $T_1, \, T_2$ have the same tree topology, except clades (possible full trees) of $T_1$ and $T_2$ differ by one NNI move shown in Fig.~\ref{fig:NNI}.  Let $X_0 = X_1 \cup X_2 \cup X_3$.  Note that the clade $T'_1$ with leaves $X_0$ in $T_1$ and the clade $T'_2$ with leaves $X_0$ in $T_2$ are equidistant trees.  Also clades with leaf sets $X_1$, $X_2$, and $X_3$ in $T_1$ and clades with leaf sets $X_1$, $X_2$, and $X_3$ in $T_2$ are also equidistant trees. 
Let $u = (u_{12}, \ldots , u_{n-1n})$ be an ultrametric associated with the tree $T_1$ and $v=(v_{12}, \ldots , v_{n-1n})$ be an ultrametric associated with the tree $T_2$.  Suppose the tree topology of $T_1$ is the top tree topology in Fig.~\ref{fig:NNI}. 
Then we have
\[
\begin{array}{c}
u_{i_1j_1} = u_{i_2j_2} \\
u_{i_1k_1} = u_{i_2k_2} \\
u_{j_1k_1} = u_{j_2k_2} \\
\end{array}
\]
for any $i_1, i_2 \in X_1$, for any $j_1, j_2 \in X_2$, and for any $k_1, k_2 \in X_3$ since the subtree consisting of $X_1$ and $X_3$ is also a clade so that it is also an equidistant tree with the root which is the most recent common ancestor of $X_1$ and $X_3$.  Thus they also satisfy the condition of ultrametrics such that
\[
\begin{array}{c}
u_{i'k'} \leq u_{i'j'} \\
u_{i'k'} \leq u_{j'k'} \\
u_{i'j'} = u_{j'k'} \\
\end{array}
\]
for any $i' \in X_1$, for any $j' \in X_2$, and for any $k' \in X_3$. 

If the tree topology of $T_2$ is the left bottom in Fig.~\ref{fig:NNI},  then
\[
\begin{array}{c}
v_{i_1j_1} = v_{i_2j_2} \\
v_{i_1k_1} = v_{i_2k_2} \\
v_{j_1k_1} = v_{j_2k_2} \\
\end{array}
\]
for any $i_1, i_2 \in X_1$, for any $j_1, j_2 \in X_2$, and for any $k_1, k_2 \in X_3$ 
since the subtree consisting of $X_1$ and $X_2$ is also a clade so that it is also an equidistant tree with the root which is the most recent common ancestor of $X_1$ and $X_2$.  Thus they also satisfy the condition of ultrametrics such that
\[
\begin{array}{c}
v_{i'j'} \leq v_{i'k'} \\
v_{i'j'} \leq v_{j'k'} \\
v_{i'k'} = v_{j'k'} \\
\end{array}
\]
for any $i' \in X_1$, for any $j' \in X_2$, and for any $k' \in X_3$. 

Similarly, if the tree topology of $T_2$ is the right bottom in Fig.~\ref{fig:NNI}, then
\[
\begin{array}{c}
v_{i_1j_1} = v_{i_2j_2} \\
v_{i_1k_1} = v_{i_2k_2} \\
v_{j_1k_1} = v_{j_2k_2} \\
\end{array}
\]
for any $i_1, i_2 \in X_1$, for any $j_1, j_2 \in X_2$, and for any $k_1, k_2 \in X_3$  
since the subtree consisting of $X_2$ and $X_3$ is also a clade so that it is also an equidistant tree with the root which is the most recent common ancestor of $X_2$ and $X_3$.  Thus they also satisfy the condition of ultrametrics such that
\[
\begin{array}{c}
v_{j'k'} \leq v_{i'k'} \\
v_{j'k'} \leq v_{i'j'} \\
v_{i'j'} = v_{i'k'} \\
\end{array}
\]
for any $i' \in X_1$, for any $j' \in X_2$, and for any $k' \in X_3$. 

Therefore, we set dummy leaves $\{i, j, k\}$ such that
\[
\begin{array}{c}
u_{ij} = u_{i_1j_1} = u_{i_2j_2} \\
u_{ik} = u_{i_1k_1} = u_{i_2k_2} \\
u_{jk} = u_{j_1k_1} = u_{j_2k_2} \\
\end{array}
\]
and
\[
\begin{array}{c}
v_{ij} = v_{i_1j_1} = v_{i_2j_2} \\
v_{ik} = v_{i_1k_1} = v_{i_2k_2} \\
v_{jk} = v_{j_1k_1} = v_{j_2k_2} \\
\end{array}
\]
then we have
\[
\begin{array}{c}
u_{ik} \leq u_{ij} \\
u_{ik} \leq u_{jk} \\
u_{ij} = u_{jk} .\\
\end{array}
\]
If the tree topology of $T_2$ is the left bottom in Fig.~\ref{fig:NNI},  then
\[
\begin{array}{c}
v_{ij} \leq v_{ik} \\
v_{ij} \leq v_{jk} \\
v_{ik} = v_{jk} .\\
\end{array}
\]
If the tree topology of $T_2$ is the right bottom in Fig.~\ref{fig:NNI},  then
\[
\begin{array}{c}
v_{jk} \leq v_{ij} \\
v_{jk} \leq v_{ik} \\
v_{ij} = v_{ik} .\\
\end{array}
\]
These reduce to the trees with three leaves $\{i, j, k\}$.  Therefore, by applying Lemma \ref{lem3leaves}, we have the result. 
\end{proof}

\begin{example}
Consider $T_1$ and $T_2$ in Example \ref{ex:subtree1}. Then, let $X_1 = \{S_1, S_2, S_3\}$, $X_2 = \{S_5\}$ and $X_3 = \{S_4\}$.  Then we notice that 
\[
u_{S_1S_4} = u_{S_2S_4} = u_{S_3S_4} = 3.8,
\]
and
\[
u_{S_1S_5} = u_{S_2S_5} = u_{S_3S_5} = 4.0.
\]
Also 
\[
v_{S_1S_4} = v_{S_2S_4} = v_{S_3S_4} = 4.0,
\]
and
\[
v_{S_1S_5} = v_{S_2S_5} = v_{S_3S_5} = 4.0.
\]
Then we introduce a leaf $i$ such that
\[
u_{iS_4} := u_{S_1S_4} = u_{S_2S_4} = u_{S_3S_4} = 3.8,
\]
\[
u_{iS_5}:= u_{S_1S_5} = u_{S_2S_5} = u_{S_3S_5} = 4.0,
\]
\[
v_{iS_4}:= v_{S_1S_4} = v_{S_2S_4} = v_{S_3S_4} = 4.0,
\]
and
\[
v_{iS_5}= v_{S_1S_5} = v_{S_2S_5} = v_{S_3S_5} = 4.0.
\]
Then we can reduce the tree with five leaves $\{S_1, S_2, S_3, S_4, S_5\}$ to the tree with three leaves $\{i, S_4, S_5\}$.  
\end{example}



\section{Discussion}

This paper is the first step toward understanding combinatorics of tree topologies along a tropical line segment between equidistant trees with $n$ leaves.  There is still so much work to be done in order to understand outputs of statistical learning models using tropical geometry over the space of equidistant trees with $n$ leaves $X$.  
It is still an open problem that we can generalize Theorem \ref{th:nni1}.  More specifically, we have the following conjecture:
\begin{conj}
Suppose we have a tropical line segment $\Gamma_{T_1, T_2}$ between equidistant trees $T_1, T_2$ with $n$ leaves $X$.  Then given tree topologies of $T_1, T_2$, the tree topology changes according to a sequence of NNI moves from $T_1$ to $T_2$ along the tropical line segment $\Gamma_{T_1, T_2}$.
\end{conj}

\begin{acknowledgements}
R.Y.~is partially supported by NSF (DMS 1916037). Also the author thank the editor and referees for improving this manuscript.  
\end{acknowledgements}

%
%

\bibliographystyle{spmpsci}      

\bibliography{3treespaces}

%
%
Author, Article title, Journal, Volume, page numbers (year)
Author, Book title, page numbers. Publisher, place (year)

\end{document}